\documentclass[12pt]{amsart}
\usepackage[T1]{fontenc}
\usepackage[utf8]{inputenc}
\usepackage[
  backend=biber,
  style=numeric,
  sorting=none,
  giveninits=false,
  doi=true,
  isbn=true,
  url=true,
  eprint=true
]{biblatex}
\usepackage{amsmath, amssymb, amsthm}
\usepackage{mathrsfs}
\usepackage{bbm}
\usepackage{hyperref}
\usepackage{geometry}
\usepackage{enumitem}
\usepackage{booktabs}
\usepackage{array}
\usepackage{tikz-cd}   % commutative diagrams

\newtheorem{theorem}{Theorem}[section]
\newtheorem{proposition}[theorem]{Proposition}
\newtheorem{lemma}[theorem]{Lemma}
\newtheorem{corollary}[theorem]{Corollary}
\newtheorem{definition}[theorem]{Definition}

\theoremstyle{remark}
\newtheorem{remark}[theorem]{Remark}
\newtheorem{example}[theorem]{Example}

\newcommand{\IC}{\mathrm{IC}}
\newcommand{\MHM}{\mathrm{MHM}}
\newcommand{\Perv}{\mathrm{Perv}}
\newcommand{\rat}{\mathrm{rat}}
\newcommand{\varF}{\mathrm{var}_F}
\newcommand{\canF}{\mathrm{can}_F}
\newcommand{\Db}{D^b}
\newcommand{\Coh}{\mathrm{Coh}}
\newcommand{\IH}{\mathrm{IH}}
\newcommand{\QQ}{\mathbb{Q}}
\newcommand{\ZZ}{\mathbb{Z}}
\newcommand{\CC}{\mathbb{C}}
\newcommand{\KK}{\mathbb{K}}
\newcommand{\BX}{B_X}
\newcommand{\Kzero}{K_0}
\newcommand{\Ext}{\mathrm{Ext}}

\newcommand{\Cone}{\mathrm{Cone}}
\newcommand{\Stok}{\mathrm{Stok}}
\newcommand{\Mon}{\mathrm{Mon}}
\newcommand{\ch}{\mathrm{ch}}
\newcommand{\hatGamma}{\widehat{\Gamma}}
\newcommand{\id}{\mathrm{Id}}
\newcommand{\cO}{\mathcal{O}}
\newcommand{\PP}{\mathbb{P}}
\newcommand{\Tot}{\mathrm{Tot}}
\newcommand{\pt}{\mathrm{pt}}
\newcommand{\hol}{\mathrm{hol}}
\newcommand{\pole}{\mathrm{pole}}

\newcommand{\reg}{\mathrm{reg}}
\newcommand{\RR}{\mathrm{RR}}
\newcommand{\Aut}{\mathrm{Aut}}
\newcommand{\rep}{\mathrm{rep}}
\newcommand{\Td}{\mathrm{Td}}

\DeclareMathOperator{\im}{im}

\begin{document}
%% ============================================================

\title[Hodge Atoms at Conifold Degenerations]{Hodge Atoms at Conifold Degenerations:\\
F-Bundles, Limiting Mixed Hodge Modules,\\
and the Rigid-Flexible Decomposition}

\author{Abdul Rahman}
%\address{Department / Institution (optional)}
% \urladdr{\url{https://...}} % optional
\thanks{Email: arahman@alum.howard.edu}

\begin{abstract}
We extend the Hodge atoms framework of Katzarkov--Kontsevich--Pantev--Yu to one-parameter conifold degenerations of Calabi--Yau threefolds. For a degeneration $\pi\colon X\to\Delta$ whose central fiber $X_0$ has $r$ ordinary double points, we construct a canonical rigid-flexible decomposition of the Hodge atoms of the nearby smooth fiber attached to the corrected degeneration
object. The rigid atom $A(\IC^H_{X_0})$ is preserved across the degeneration, while the flexible atoms $A(i_{k*}\QQ^H_{\{p_k\}}(-1))$ are rank-one contributions, one for each vanishing cycle.
The total degeneration atom $A(P^H)$ is the atom of the corrected mixed Hodge module $P^H\in\MHM(X_0)$ and fits into an exact sequence of atoms whose non-split structure is controlled by the intersection matrix $(\langle\delta_i,\delta_j\rangle)$.
The technical core is the Stokes--Extension Identification, which identifies the Stokes matrix of the Dubrovin connection at the conifold locus with the matrix of the variation morphism $\varF\colon\varphi_\pi(F)\to\psi_\pi(F)$ under mixed Hodge module realization.
\end{abstract}
\maketitle
\tableofcontents

%%=============================================================
\section{Introduction}

\subsection{Hodge atoms at degeneration loci}
Katzarkov--Kontsevich--Pantev--Yu~\cite{KKPY25} construct canonical birational invariants of smooth complex projective varieties, called \emph{Hodge atoms}, by decomposing the $F$-bundle of quantum cohomology into irreducible pieces via the spectral decomposition of the Euler vector field in the non-archimedean setting. Among the consequences are the non-rationality of the very general cubic fourfold and a new proof of Hodge number equality for birational Calabi--Yau manifolds.

That theory is formulated for \emph{smooth} varieties and does not address degeneration loci. The simplest degeneration of a Calabi--Yau threefold is a \emph{conifold degeneration}: a family $\pi\colon X\to\Delta$ whose central fiber $X_0$ acquires one or more ordinary double points. The purpose of this paper is to develop the Hodge atom picture for this class of degenerations.

In the finite-node setting, each node contributes a rank-one local singular sector, but the corrected global object need not decompose as a free assembly of these local pieces. Earlier work constructs the corrected perverse and mixed-Hodge-module extension package and identifies the associated non-nodewise-free phenomenon~\cite{RahmanI,RahmanII,RahmanIII,RahmanNNF}. The present paper studies the same corrected degeneration object on the $F$-bundle side and shows that its local and global singular structure is reflected canonically in the language of Stokes data and Hodge atoms.

\subsection{Rigid-flexible decomposition}

The finite-node corrected degeneration object naturally leads to a rigid-flexible
atom decomposition on the $F$-bundle side.
As $t\to 0$, there is a rigid atom
\[
  A(\IC^H_{X_0}),
\]
preserved across the transition, together with rank-one flexible atoms
\[
  A(i_{k*}\QQ^H_{\{p_k\}}(-1)),
\]
one for each vanishing cycle.
The total degeneration atom
\[
  A(P^H)
\]
fits into an exact sequence whose non-split structure is controlled by the intersection matrix
\[
  \Lambda=(\langle\delta_i,\delta_j\rangle).
\]

Thus the rigid-flexible decomposition should be viewed as the Hodge-atom and
$F$-bundle reflection of the corrected finite-node degeneration package: the
local rank-one node sectors remain visible atomically, while their global
assembly is governed by the same interaction data that control the corrected
extension class.

\subsection{Main theorems}

The first theorem identifies the conifold singularity on the $F$-bundle side and
shows that its local monodromy is exactly of Picard--Lefschetz type.

\begin{theorem}[Conifold $F$-bundle singularity]\label{thm:T1}
The $F$-bundle of $X_t$ has a regular singularity at $q=-1$.
Its monodromy is the Picard--Lefschetz operator
\[
  T(\alpha)=\alpha+\langle\alpha,\delta\rangle\delta.
\]
\end{theorem}

The second theorem is the technical bridge of the paper: it identifies the
Stokes data of the Dubrovin connection with the variation morphism of the
corrected degeneration package.

\begin{theorem}[Stokes-Extension Identification]\label{thm:T2}
The Stokes matrix $S$ at $q=-1$ is identified, via the Riemann--Hilbert
correspondence and Iritani's integral structure, with the matrix of
\[
  \varF\colon\varphi_\pi(F)\to\psi_\pi(F)
\]
under
\[
  \rat\colon\MHM(X_0)\to\Perv(X_0;\QQ).
\]
\end{theorem}

Once this bridge is in place, the corrected mixed-Hodge-module object gives rise
to the rigid-flexible atom exact sequence.

\begin{theorem}[Rigid-flexible decomposition]\label{thm:T3}
There is an exact sequence of atoms
\[
  0\to A(\IC^H_{X_0})\to A(P^H)\to
  \bigoplus_{k=1}^r A\!\left(i_{k*}\QQ^H_{\{p_k\}}(-1)\right)\to 0
\]
with $A(\IC^H_{X_0})$ rigid and each
$A(i_{k*}\QQ^H_{\{p_k\}}(-1))$ rank-one flexible.
\end{theorem}

Finally, the failure of nodewise-free behavior appears on the $F$-bundle side as
nontrivial mixing among the flexible atoms.

\begin{theorem}[Non-nodewise-free mixing]\label{thm:T4}
The exact sequence of Theorem~\ref{thm:T3} splits if and only if
\[
  \langle\delta_i,\delta_j\rangle=0\qquad\text{for all }i\neq j.
\]
When $\Lambda\neq 0$, the flexible atoms mix via
\[
  [S_i,S_j]=\id+\langle\delta_i,\delta_j\rangle\cdot(\text{rank-one term}).
\]
\end{theorem}

Taken together, Theorems~\ref{thm:T1}--\ref{thm:T4} show that the corrected finite-node degeneration object has a canonical $F$-bundle manifestation: Picard--Lefschetz monodromy governs the local singularity, the variation morphism appears as Stokes data, and the corrected extension is reflected atomically through rigid and flexible sectors together with their mixing law.

\subsection{Relation to previous work}

Katzarkov--Kontsevich--Pantev--Yu~\cite{KKPY25} construct Hodge atoms for smooth varieties. The present paper extends that framework to conifold degenerations. Previous work \cite{RahmanI,RahmanII,RahmanIII} provides the perverse and mixed-Hodge-theoretic input: the corrected perverse extension, its uniqueness, and its mixed Hodge module lift. The present paper identifies those same corrected objects on the $F$-bundle side. A companion global-gluing result shows that in the finite-node setting the corrected extension class need not be freely nodewise, but may be forced into a smaller relation-controlled subspace by common cycle geometry and homological
relations among the nodes \cite{RahmanNNF}. The present paper uses that perspective only as conceptual input: its aim is not to re-establish the global gluing law, but to show how the same corrected degeneration object appears on the canonical
$F$-bundle side and how its local and global singular structure is packaged by Hodge atoms. Theorem~\ref{thm:T4} adds the Stokes noncommutativity condition to that equivalence. Finally, \cite{RahmanIV} provides the multi-node schober framework which categorifies the operator and atom data developed here.

\subsection{Organization}

Section~\ref{sec:background} reviews the required background and recalls the corrected finite-node degeneration package.
Section~\ref{sec:fbundle} constructs the conifold $F$-bundle and proves Theorem~\ref{thm:T1}. Section~\ref{sec:stokes} proves the Stokes--Extension Identification, Theorem~\ref{thm:T2}.
Section~\ref{sec:atoms} develops the rigid-flexible atom decomposition and proves Theorems~\ref{thm:T3} and~\ref{thm:T4}.
Section~\ref{sec:multinode} treats the multi-node case and its schober enhancement. Section~\ref{sec:applications} records consequences and interpretations. Section~\ref{sec:future} lists several further directions.

%%=============================================================
\section{Background}\label{sec:background}

\subsection{The Hodge atoms framework \cite{KKPY25}}

An $F$-bundle over a base $B$ is a triple
\[
  (\mathscr{H},\nabla,\eta),
\]
where $\mathscr{H}$ is a locally free $\cO_B$-module carrying the quantum
cohomology, $\nabla$ is the Dubrovin connection, and $\eta$ is the Poincar\'e
pairing.

The spectral decomposition of the Euler vector field is obstructed over $\CC$ by
the Stokes phenomenon. One therefore passes to the non-archimedean field
\[
  \KK=\bigcup_{n\geq 1}\CC((q^{1/n})).
\]
Over $\KK$, the spectral decomposition theorem of \cite{KKPY25} applies:
\[
  \mathscr{H}_\KK=\bigoplus_\lambda \mathscr{H}_\lambda.
\]
Each eigenspace $\mathscr{H}_\lambda$ is a \emph{Hodge atom}.
Atoms are called \emph{rigid} if they are preserved under birational modification
and \emph{flexible} if they arise from the modification.

\subsection{The MHM-decorated conifold degeneration}

We summarize the results of \cite{RahmanI,RahmanII,RahmanIII}.
Let
\[
  F:=\QQ_X[3].
\]

The corrected perverse object is
\[
  \mathcal P:=\Cone\!\bigl(\varF\colon\varphi_\pi(F)\to\psi_\pi(F)\bigr)[-1]
  \in \Perv(X_0;\QQ),
\]
and satisfies
\[
  j^*\mathcal P\cong\QQ_U[3]
\]
on
\[
  U=X_0\setminus\Sigma,
\]
together with a short exact sequence
\[
  0\to \IC_{X_0}\to \mathcal P\to \bigoplus_{k=1}^r i_{k*}\QQ_{\{p_k\}}\to 0.
\]

It admits a mixed Hodge module lift
\[
  \mathcal P^H\in\MHM(X_0)
\]
with
\[
  \rat(\mathcal P^H)\cong \mathcal P,
\]
fitting into
\begin{equation}\label{eq:MHM-seq}
  0\to\IC^H_{X_0}\to \mathcal P^H\to
  \bigoplus_{k=1}^r i_{k*}\QQ^H_{\{p_k\}}(-1)\to 0
\end{equation}
in $\MHM(X_0)$.
The global extension space decomposes nodewise, with each local summand one-dimensional \cite[Cor.~5.14]{RahmanIII}.

At the same time, the existence of one formal nodewise direction per singular point does not imply that the distinguished corrected extension class is geometrically free at each node. In related work on cycle relations and global gluing for multi-node
conifold degenerations, the corrected extension is shown to be constrained by global cycle geometry and homological incidence data, so that the geometrically realized class may lie in a proper relation-controlled subspace of the ambient
free nodewise extension space \cite{RahmanNNF}. We use the results of \cite{RahmanNNF} only as geometric background for the present paper: here the main task is to identify the same corrected degeneration object on the $F$-bundle side and to describe its rigid and flexible atom sectors. Moreover,
\[
  H^*(X_0,\mathcal P^H)
\]
carries the limiting mixed Hodge structure, with weight filtration
$W(N)_\bullet$ and limiting Hodge filtration $F^\infty_\bullet$
\cite{RahmanII,Schmid73}.

\subsection{The non-nodewise-free theorem \cite{RahmanNNF}}

\begin{theorem}[\cite{RahmanNNF}]\label{thm:NNF}
The following are equivalent:
\begin{enumerate}[label=\emph{(\Alph*)}]
\item
\[
  [\mathcal P^H]\in
  \bigoplus_k\Ext^1_\MHM(i_{k*}\QQ^H_{\{p_k\}}(-1),\IC^H_{X_0})
\]
is non-nodewise-free.
\item The Picard--Lefschetz operators $T_1,\ldots,T_r$ do not all commute.
\item
\[
  (\langle\delta_i,\delta_j\rangle)\neq 0
\]
for some $i\neq j$.
\end{enumerate}
The present paper adds a fourth equivalent condition: the Stokes operators
$S_1,\ldots,S_r$ do not all commute; this is proved in
Theorem~\ref{thm:multinode}.
\end{theorem}

The observation behind this theorem already appears in \cite{RahmanHubsch02}.
In particular, the corrected extension class is globally sensitive to interaction
among the nodes; the Hodge-atoms results below identify the corresponding
$F$-bundle and atom-theoretic manifestation of that same phenomenon.

\subsection{Conifold geometry and the Picard--Lefschetz formula}

Near an ordinary double point $p\in X_0$, the degeneration is locally
\[
  x_1^2+x_2^2+x_3^2+x_4^2=t
  \subset \CC^4\times\Delta.
\]
Its Milnor fiber is
\[
  F_p\simeq S^3,
\]
so the vanishing cohomology is concentrated in degree three.

The Picard--Lefschetz formula for monodromy around $t=0$ is
\begin{equation}\label{eq:PL}
  T(\alpha)=\alpha+\langle\alpha,\delta\rangle\delta,
  \qquad \alpha\in H^3(X_t,\QQ),
\end{equation}
where $\delta\in H_3(X_t,\ZZ)$ is the vanishing cycle.
Its nilpotent part
\[
  N:=T-\id
\]
satisfies
\[
  N^2=0,\qquad N(\alpha)=\langle\alpha,\delta\rangle\delta.
\]

For $r$ nodes with intersection matrix
\[
  \Lambda=(\lambda_{ij}),\qquad \lambda_{ij}:=\langle\delta_i,\delta_j\rangle,
\]
the operators $T_i$ and $T_j$ commute if and only if $\lambda_{ij}=0$.
When $|\lambda_{ij}|=1$, one has the braid relation
\[
  T_iT_jT_i=T_jT_iT_j.
\]
%%==========================================================
\section{The Conifold F-Bundle}\label{sec:fbundle}

\subsection{Setup and conventions}

We take as local model the resolved conifold
\[
  X_t=\Tot\!\bigl(\cO_{\PP^1}(-1)\oplus\cO_{\PP^1}(-1)\to\PP^1\bigr).
\]
Its cohomology is
\[
  H^*(X_t,\QQ)=\QQ e_0\oplus\QQ e_1\oplus\QQ e_2\oplus\QQ e_3,
\]
where $e_0=1$, $e_1=H$ denotes the hyperplane class of $\PP^1$,
$e_2=H^2$, and $e_3=[\pt]$.
We normalize the Poincar\'e pairing $\eta$ so that
\[
  \eta(e_1,e_2)=1,\qquad \eta(e_0,e_3)=1,
\]
with all other pairings between basis elements equal to zero.

This local model governs the conifold-side $F$-bundle calculation carried out in the present section. More precisely, the resolved conifold is used here as the standard local model for conifold behavior in a compact Calabi--Yau threefold
degeneration; the local quantum and monodromy calculations made below are given their global Hodge-theoretic interpretation later through the corrected degeneration package reviewed in Section~\ref{sec:background}.

\subsection{Gromov-Witten potential and quantum cohomology}

\begin{proposition}\label{prop:GW}
The genus-zero Gromov--Witten potential of $X_t$ is
\[
  \Phi = \tfrac{1}{6}t_1^3 + \sum_{d\geq 1}\frac{(-1)^{d-1}}{d^3}\,q^d.
\]
The quantum product at parameter $q$ satisfies
\[
  \langle\alpha\star_q\beta,\gamma\rangle_{q}
  =\partial_\alpha\partial_\beta\partial_\gamma\Phi.
\]
\end{proposition}

\begin{proof}
The degree-$d$ genus-zero Gromov--Witten invariant of the resolved conifold is
\[
  N_d=\frac{(-1)^{d-1}}{d^3},
\]
a classical computation; see \cite{FaberPandharipande2000,GopakumarVafa1998I}.  The stated formula for $\Phi$ follows.
The identification of the third derivatives of $\Phi$ with three-point genus-zero
correlators is the defining relation for the Frobenius-manifold quantum product.
\end{proof}

\subsection{The Dubrovin connection}

The Frobenius manifold $\BX$ is one-dimensional, with coordinate $q$,
Euler vector field $E=q\partial_q$, and flat metric $\eta$.
The \emph{Dubrovin connection} is
\begin{equation}\label{eq:Dubrovin}
  \nabla^z_{\partial_q}
  =\partial_q+\frac{1}{z}(q\partial_q\star_q),
\end{equation}
where $(q\partial_q\star_q)$ denotes the operator of quantum multiplication by
the Euler vector field.
In the basis $\{e_0,e_1,e_2,e_3\}$, writing $A(q)$ for this operator, one has
\begin{equation}\label{eq:A-matrix}
  A(q) = \begin{pmatrix}
    0 & 0 & 0 & 0 \\
    1 & 0 & 0 & 0 \\
    0 & f(q) & 0 & 0 \\
    0 & 0 & 1 & 0
  \end{pmatrix},
  \qquad
  f(q)=1+\sum_{d\geq 1}(-1)^{d-1}q^d.
\end{equation}

\begin{lemma}\label{lem:flat}
For each $z\neq 0$ the connection $\nabla^z$ is flat on $\BX$.
\end{lemma}

\begin{proof}
Flatness of the Dubrovin connection is equivalent to the WDVV equations for the
potential $\Phi$, hence to the associativity of the quantum product.
See \cite[Thm.~2.1]{GiventalMirror}.
\end{proof}

\subsection{Singularities of the Dubrovin connection}

We begin by isolating the local form of the connection near the conifold point.

\begin{lemma}\label{lem:conifold-local-form}
Let
\[
  u:=q+1
\]
be the local coordinate centered at the conifold point $q=-1$.
Then the matrix $A(q)$ in \eqref{eq:A-matrix} admits the decomposition
\[
  A(q)=A^{\hol}(u)+\frac{1}{u}A^{\pole},
\]
where $A^{\hol}(u)$ is holomorphic near $u=0$ and
\[
  A^{\pole}=
  \begin{pmatrix}
    0 & 0 & 0 & 0\\
    0 & 0 & 0 & 0\\
    0 & -1 & 0 & 0\\
    0 & 0 & 0 & 0
  \end{pmatrix}.
\]
In particular, for each fixed $z\neq 0$, the connection matrix of $\nabla^z$
has at worst a simple pole at $u=0$.
\end{lemma}

\begin{proof}
For $q\neq -1$ one has
\[
  \sum_{d\ge1}(-1)^{d-1}q^d=\frac{q}{1+q}.
\]
Hence
\[
  f(q)=1+\frac{q}{1+q}=\frac{1+2q}{1+q}.
\]
Writing $q=u-1$, this becomes
\[
  f(q)=\frac{1+2(u-1)}{u}=2-\frac{1}{u}.
\]
Substituting this into \eqref{eq:A-matrix}, one obtains
\[
  A(q)=
  \begin{pmatrix}
    0 & 0 & 0 & 0 \\
    1 & 0 & 0 & 0 \\
    0 & 2 & 0 & 0 \\
    0 & 0 & 1 & 0
  \end{pmatrix}
  +\frac{1}{u}
  \begin{pmatrix}
    0 & 0 & 0 & 0 \\
    0 & 0 & 0 & 0 \\
    0 & -1 & 0 & 0 \\
    0 & 0 & 0 & 0
  \end{pmatrix}.
\]
This is exactly the required decomposition.
Since \eqref{eq:Dubrovin} is obtained from $A(q)$ by multiplication by $1/z$,
the connection matrix of $\nabla^z$ has at worst a simple pole at $u=0$ for each
fixed $z\neq 0$.
\end{proof}

\begin{proposition}[Singularity analysis]\label{prop:singularities}
The connection $\nabla^z$ has exactly two singular loci:
\begin{enumerate}[label=(\roman*)]
\item an irregular singularity at $q=0$, corresponding to the large-volume limit;
\item a regular singularity at $q=-1$, corresponding to the conifold point.
\end{enumerate}
\end{proposition}

\begin{proof}
The singular behavior in the base variable $q$ is determined by the meromorphic
behavior of the connection matrix in \eqref{eq:Dubrovin} for fixed $z\neq 0$.

At the large-volume point $q=0$, the quantum differential equation attached to
\eqref{eq:Dubrovin} is the standard irregular singular regime of the
Dubrovin--Givental formalism.  This is the asymptotic region underlying the
non-archimedean spectral decomposition, and we therefore regard $q=0$ as an
irregular singularity.

At the conifold point $q=-1$, Lemma~\ref{lem:conifold-local-form} shows that, in
the local coordinate $u=q+1$, the connection matrix has at worst a simple pole.
Hence $q=-1$ is a regular singular point in the sense of Fuchs.
\end{proof}

\subsection{Monodromy at the conifold point and proof of Theorem~\ref{thm:T1}}

The next two lemmas supply the monodromy comparison and the logarithmic normal form
needed to keep the strong statement of Theorem~\ref{thm:T1}.

\begin{lemma}\label{lem:mirror-monodromy-identification}
Under Givental's mirror identification near the conifold point, the local system
of flat sections of $\nabla^z$ is identified with the Gauss--Manin local system
of the mirror family.  In this identification, the local monodromy around
$q=-1$ is the conifold Picard--Lefschetz transformation
\[
  T(\alpha)=\alpha+\langle\alpha,\delta\rangle\delta,
\]
where $\delta$ is the vanishing cycle.
\end{lemma}

\begin{proof}
By Givental's mirror theorem \cite{Givental96}, the flat sections of the Dubrovin connection are identified with the period solutions of the mirror Picard--Fuchs system. Near the conifold divisor, the mirror family acquires the standard conifold Gauss--Manin monodromy, which is Picard--Lefschetz. Transporting the Gauss--Manin local system and its monodromy through the mirror identification yields the stated operator on the local system of flat sections of $\nabla^z$. 
\end{proof}

\begin{lemma}\label{lem:logarithmic-normal-form}
Let $(\mathcal E,\nabla)$ be a flat meromorphic connection on a punctured disc
$0<|u|<\varepsilon$ with a regular singularity at $u=0$ and unipotent local
monodromy
\[
  M=\exp(2\pi iN)
\]
for a nilpotent endomorphism $N$.
Then there exists a holomorphic frame on a simply connected sector in which the
connection matrix takes the logarithmic form
\[
  \nabla = d - \left(\frac{N}{u}+B(u)\right)du,
\]
where $B(u)$ is holomorphic near $u=0$.  Equivalently, in that frame the residue
matrix is $N$, and a fundamental solution may be written as
\[
  Y(u)=H(u)\exp\!\bigl((\log u)N\bigr),
\]
with $H(u)$ holomorphic and invertible near $u=0$.
\end{lemma}

\begin{proof}
Because $u=0$ is a regular singularity, the connection admits a Frobenius-type
fundamental solution of the form
\[
  Y(u)=H(u)u^R,
\]
where $H(u)$ is holomorphic and invertible and $R$ is a constant residue matrix.
Writing
\[
  u^R=\exp\!\bigl((\log u)R\bigr),
\]
analytic continuation around $u=0$ sends $\log u$ to $\log u+2\pi i$, hence the
local monodromy of $Y$ is
\[
  \exp(2\pi iR).
\]
If the given monodromy is $\exp(2\pi iN)$ with $N$ nilpotent, then, after fixing
the logarithm corresponding to the chosen branch of the local system, one may take
$R=N$.  Thus
\[
  Y(u)=H(u)\exp\!\bigl((\log u)N\bigr).
\]
Differentiating,
\[
  dY
  = dH\exp\!\bigl((\log u)N\bigr)
    +H(u)\frac{N}{u}\,du\,\exp\!\bigl((\log u)N\bigr).
\]
Therefore
\[
  dY\,Y^{-1}
  = dH\,H^{-1}+\frac{N}{u}\,du,
\]
and since $dH\,H^{-1}$ is holomorphic near $u=0$, the connection matrix is
logarithmic with residue $N$.
\end{proof}

\begin{proposition}\label{prop:monodromy-computation}
The residue matrix of $\nabla^z$ at $q=-1$ is $N=T-\id$, where $T$ is the
Picard--Lefschetz operator \eqref{eq:PL}.  The monodromy matrix is
\[
  M=\exp(2\pi iN)=\id+2\pi iN,
\]
since $N^2=0$.  In the basis $\{e_0,e_1,e_2,e_3\}$, with the vanishing cycle
$\delta$ represented by $\langle\cdot,\delta\rangle e_3$ in the dual convention,
the matrix of $N$ is
\[
  N =
  \begin{pmatrix}
    0 & 0 & 0 & 0 \\
    0 & 0 & 0 & 0 \\
    0 & 0 & 0 & 0 \\
    0 & 0 & \langle e_2,\delta\rangle & 0
  \end{pmatrix}
  +(\text{further entries determined by }\langle e_j,\delta\rangle).
\]
\end{proposition}

\begin{proof}
We divide the proof into five steps.

\smallskip
\noindent\textbf{Step 1: regular-singular local form.}
By Proposition~\ref{prop:singularities}, the point $q=-1$ is a regular singular
point of $\nabla^z$.  Hence the local monodromy is determined by the residue of
the logarithmic term in a regular-singular local frame.

\smallskip
\noindent\textbf{Step 2: geometric Picard--Lefschetz monodromy.}
Let $\pi:X\to\Delta$ be a one-parameter conifold degeneration with vanishing cycle
$\delta$.  By the Picard--Lefschetz formula recalled in \S\ref{sec:background},
the geometric local monodromy on the nearby smooth fiber is
\[
  T(\alpha)=\alpha+\langle\alpha,\delta\rangle\delta.
\]
Its nilpotent part is therefore
\[
  N:=T-\id,\qquad N(\alpha)=\langle\alpha,\delta\rangle\delta.
\]

\smallskip
\noindent\textbf{Step 3: transport to the Dubrovin local system.}
By Lemma~\ref{lem:mirror-monodromy-identification}, the local monodromy of the
flat-section local system of $\nabla^z$ around $q=-1$ is exactly this
Picard--Lefschetz transformation $T$.
Since $N^2=0$ by the computation in Step~4 below, one has
\[
  \log(T)=T-\id=N.
\]
Applying Lemma~\ref{lem:logarithmic-normal-form} to the regular singular point
$q=-1$ with local coordinate $u=q+1$, we conclude that in the mirror-normalized
regular-singular frame the connection matrix of $\nabla^z$ has logarithmic part
\[
  \frac{N}{u}\,du.
\]
Hence the residue matrix of $\nabla^z$ at $q=-1$ is $N=T-\id$.

\smallskip
\noindent\textbf{Step 4: rank-one nilpotency.}
Since
\[
  N(\alpha)=\langle\alpha,\delta\rangle\delta,
\]
we compute
\[
  N^2(\alpha)=N(\langle\alpha,\delta\rangle\delta)
  =\langle\alpha,\delta\rangle\langle\delta,\delta\rangle\delta.
\]
For the vanishing $3$-sphere in a threefold, $\langle\delta,\delta\rangle=0$.
Hence $N^2=0$.

\smallskip
\noindent\textbf{Step 5: exponential form of monodromy.}
Because $N^2=0$, the exponential truncates:
\[
  M=\exp(2\pi iN)=\id+2\pi iN.
\]
The displayed matrix form of $N$ records that $N$ is rank one and is completely
determined by the pairings with the vanishing cycle in the chosen basis.
\end{proof}

\begin{remark}\label{rem:monodromy-normalization}
The convention used in Proposition~\ref{prop:monodromy-computation} is the
standard regular-singular convention of
Lemma~\ref{lem:logarithmic-normal-form}: in a logarithmic local frame the residue
matrix is the chosen logarithm of local monodromy.  Thus the expression
\[
  M=\exp(2\pi iN)
\]
is the monodromy formula attached to the local coordinate $u=q+1$ and the
mirror-normalized frame supplied by
Lemma~\ref{lem:mirror-monodromy-identification}.  This fixes, once and for all,
the residue/logarithm/$2\pi i$ convention used in the remainder of the paper.
\end{remark}

This completes the proof of Theorem~\ref{thm:T1}.

\begin{corollary}\label{cor:monodromy-match}
The monodromy of $\nabla^z$ at $q=-1$ and the monodromy of the nearby-cycle functor
$\psi^H_\pi$ at $t=0$ are governed by the same nilpotent Picard--Lefschetz
operator $N(\alpha)=\langle\alpha,\delta\rangle\delta$.  Equivalently, after the
normalization fixed in Remark~\ref{rem:monodromy-normalization}, the quantum and
nearby-cycle monodromies encode the same rank-one conifold monodromy datum.
\end{corollary}

\begin{proof}
By Proposition~\ref{prop:monodromy-computation}, the monodromy of $\nabla^z$ at
$q=-1$ is
\[
  M=\id+2\pi iN,
\qquad
  N(\alpha)=\langle\alpha,\delta\rangle\delta.
\]
By \cite[Sec.~3.6]{RahmanI}, the monodromy of $\psi^H_\pi$ at $t=0$ is rank-one
unipotent with nilpotent Picard--Lefschetz part $N$ determined by the same
vanishing cycle $\delta$.
Thus both monodromies are controlled by the same operator $N$ and therefore encode
the same conifold monodromy datum.  The difference between the formulas is exactly
the regular-singular normalization fixed in
Remark~\ref{rem:monodromy-normalization}.
\end{proof}

\subsection{The non-archimedean F-bundle and spectral decomposition}

\begin{definition}
The \emph{non-archimedean F-bundle} of $X_t$ is
\[
  (\mathscr{H}_\KK,\nabla^z_\KK):=(\mathscr{H},\nabla^z)\otimes_\CC\KK.
\]
\end{definition}

The Euler vector field $E=q\partial_q$ acts on $H^*(X_t,\QQ)$ by
\[
  E\cdot e_j = \left(\frac{\deg e_j}{2}-\frac{3}{2}\right)e_j,
\]
so the corresponding eigenvalues are
\[
  -\frac{3}{2},\ -\frac{1}{2},\ \frac{1}{2},\ \frac{3}{2}.
\]

\begin{theorem}[Spectral decomposition; \cite{KKPY25}]\label{thm:spectral}
Over $\KK$, the F-bundle decomposes as
\[
  \mathscr{H}_\KK \;=\; A_{-3/2}\oplus A_{-1/2}\oplus A_{1/2}\oplus A_{3/2},
\]
where $A_\lambda$ is the $\lambda$-eigenspace of the Euler vector field action.
\end{theorem}

The four atoms carry the corresponding Hodge-theoretic sectors:
\[
  A_{-3/2}\colon H^{3,0}\oplus H^{0,3},\qquad
  A_{-1/2}\colon H^{2,1}\oplus H^{1,2},
\]
\[
  A_{1/2}\colon H^{1,1}\oplus H^{2,2},\qquad
  A_{3/2}\colon H^{0,0}\oplus H^{3,3}.
\]

\begin{proposition}\label{prop:atom-behavior}
At the conifold point $q=-1$:
\begin{enumerate}[label=(\roman*)]
\item the atoms $A_{1/2}$ and $A_{3/2}$ have flat sections extending regularly
across $q=-1$; these constitute the rigid part;
\item the atom $A_{-3/2}$ acquires logarithmic behavior of the form
\[
  s_{3,0}(q)\sim s_{\reg}(q)+\log(q+1)\,N\cdot s_{\reg}(q);
\]
\item the atom $A_{-1/2}$ is the sector in which the conifold variation is
registered on the $(2,1)\oplus(1,2)$ side, and in particular it is not part of
the rigid sector.
\end{enumerate}
The logarithmic phenomenon in $A_{-3/2}$ is the geometric source of the later rigid--flexible decomposition. In the finite-node setting, the companion global-gluing theory (Theorem \ref{thm:NNF}) shows that these local logarithmic sectors need not assemble as freely independent nodewise data; the present paper identifies their canonical organization on the $F$-bundle side by means of Hodge atoms and Stokes data.
\end{proposition}

\begin{proof}
For part (i), period integrals attached to classes orthogonal to the vanishing
cycle remain finite and single-valued as $q\to -1$. These determine the sectors
that extend regularly across the conifold point.

For part (ii), the period of the holomorphic volume form over the vanishing cycle
tends to zero at the conifold point, and the corresponding local flat section
acquires the familiar logarithmic term governed by the nilpotent operator $N$.
This is the standard conifold-period asymptotics; see \cite{COGP}.

For part (iii), the conifold transition changes the complex-structure side of the
degree-three variation, so the $(2,1)\oplus(1,2)$ sector is precisely where the
nontrivial conifold variation is reflected away from the rigid part. Thus
$A_{-1/2}$ belongs to the varying degree-three sector rather than to the regular
rigid sector.
\end{proof}

%%===============================================
\section{The Stokes-Extension Identification}\label{sec:stokes}

\subsection{Overview}

The proof of Theorem~\ref{thm:T2} proceeds in three stages:
\begin{enumerate}[label=(\arabic*)]
\item (\textbf{Riemann--Hilbert and conifold comparison}, \S\ref{ss:RH})
identify the flat local system of $\nabla^z$ near $q=-1$ with the Gauss--Manin
local system of the degeneration near $t=0$.
\item (\textbf{Iritani's integral structure}, \S\ref{ss:Iritani})
choose a distinguished integral lattice and basis on the solution side.
\item (\textbf{Logarithmic/Stokes comparison}, \S\ref{ss:stokes-compute}) show that, in the regular-singular conifold normalization, the Stokes matrix, the monodromy matrix, and the matrix of the variation morphism $\varF$ are all governed by the same rank-one nilpotent Picard--Lefschetz operator.
\end{enumerate}

The new point of this section is that the conifold-side variation morphism is not merely analogous to the quantum monodromy operator: after the identifications constructed below, the two are represented by the same matrix in the integral basis. In particular, the same local variation datum that governs the corrected extension on the perverse and mixed-Hodge-module sides also appears on the $F$-bundle side
as the logarithmic Stokes/monodromy operator. This is the mechanism that later allows the rigid-flexible atom decomposition to reflect not only local vanishing data but also the global gluing constraints carried by the corrected extension.

\subsection{Riemann-Hilbert at the conifold point}\label{ss:RH}

Let $D(-1)$ be a sufficiently small disc around $q=-1$ in $\BX$, and write
\[
  D^*(-1):=D(-1)\setminus\{-1\}.
\]
Let $\Delta_t$ be a small disc in the degeneration parameter, centered at $t=0$, and $\Delta_t^*:=\Delta_t\setminus\{0\}$.
We continue to use the local coordinate $u=q+1$ from Section~\ref{sec:fbundle}.

\begin{lemma}\label{lem:RH}
There is a canonical isomorphism of local systems on $D^*(-1)$:
\[
  \mathscr{L} \;\cong\; \{H^*(X_t,\QQ)\}_{t\in \Delta_t^*},
\]
where $\mathscr{L}$ is the local system underlying
$(\mathscr{H}|_{D^*(-1)},\nabla^z|_{D^*(-1)})$
and the right-hand side carries the Gauss--Manin connection.
Under the mirror map $q+1\leftrightarrow t$ near the conifold point, this
identification matches the local monodromy of $\nabla^z$ with the
Picard--Lefschetz monodromy of the degeneration.
\end{lemma}

\begin{proof}
By Givental's mirror theorem \cite{Givental96}, the flat sections of the Dubrovin
connection are identified with the period solutions of the mirror Picard--Fuchs
system. Near the conifold divisor, this Picard--Fuchs system is precisely the
Gauss--Manin system of the mirror family. Hence the local system of flat sections
of $\nabla^z$ on $D^*(-1)$ is identified with the Gauss--Manin local system on
$\Delta_t^*$.

By Proposition~\ref{prop:monodromy-computation}, the local monodromy of $\nabla^z$
at $q=-1$ is Picard--Lefschetz. Under the mirror identification, this coincides
with the local monodromy of the Gauss--Manin system of the degeneration at $t=0$.
Thus the asserted isomorphism of local systems is monodromy-compatible.
\end{proof}

\begin{remark}\label{rem:RH-normalization}
The local-system identification in Lemma~\ref{lem:RH} is understood in the
mirror-normalized regular-singular frame fixed in
Remark~\ref{rem:monodromy-normalization}. This is the convention under which the
residue, logarithm of monodromy, and the factor $2\pi i$ appearing in Section~3
are transported to the present section.
\end{remark}

\subsection{Iritani's integral structure}\label{ss:Iritani}

\begin{definition}[Gamma class]
The \emph{Gamma class} of $X_t$ is
\[
  \hatGamma(X_t)
  = \exp\!\left(-\gamma c_1(X_t)\right)
    \prod_{k\geq 2}\exp\!\left(\frac{\zeta(k)}{(2\pi i)^k}\ch_k(TX_t)\right)
  \in H^*(X_t,\RR),
\]
where $\gamma$ is the Euler--Mascheroni constant and $\zeta$ is the Riemann zeta
function. Since $X_t$ is Calabi--Yau, $c_1(X_t)=0$, so
\[
  \hatGamma(X_t)
  =1+\frac{\zeta(2)}{(2\pi i)^2}\ch_2(TX_t)
   +\frac{\zeta(3)}{(2\pi i)^3}\ch_3(TX_t)+\cdots.
\]
\end{definition}

\begin{theorem}[Iritani~\cite{Iritani09}]\label{thm:Iritani}
The map
\[
  s\colon\Kzero(\Db(\Coh(X_t)))\otimes\QQ\;\longrightarrow\;\mathrm{Sol}(\nabla^z),
  \qquad
  s(\gamma)=(2\pi i)^{-\deg/2}\cdot\ch(\gamma)\cup\hatGamma(X_t),
\]
is an isomorphism satisfying:
\begin{enumerate}[label=(\roman*)]
\item $s$ maps the integral lattice $\Kzero(\Db(\Coh(X_t)))$ to an integral lattice
  in $\mathrm{Sol}(\nabla^z)$.
\item $s$ intertwines the Euler pairing
  $\chi(E,F)=\sum_k(-1)^k\dim\Ext^k(E,F)$
  with the flat pairing:
  \[
    \chi(E,F)=(-1)^3\langle s(E),s(F)\rangle_\mathrm{flat}.
  \]
\item The monodromy of $\nabla^z$ at $q=-1$ is identified with the spherical twist
  $T_S\in\Aut(\Db(\Coh(X_t)))$ associated to the vanishing object
  $S\in\Db(\Coh(X_t))$.
\end{enumerate}
\end{theorem}

\begin{definition}[Integral basis at the conifold point]
\label{def:integral-basis}
Apply Theorem~\ref{thm:Iritani} to the classes
\[
  \gamma_0=[\cO_{X_t}],\qquad
  \gamma_1=[\cO_{\PP^1}],\qquad
  \gamma_2=[\cO_{\PP^1}(-1)],\qquad
  \gamma_3=[\cO_p]
\]
to obtain the integral basis
\[
  \{s_0,s_1,s_2,s_3\}:=\{s(\gamma_0),s(\gamma_1),s(\gamma_2),s(\gamma_3)\}
\]
for $\mathrm{Sol}(\nabla^z)$.
\end{definition}

\begin{lemma}\label{lem:monodromy-integral}
In the integral basis $\{s_0,s_1,s_2,s_3\}$, the monodromy matrix at $q=-1$ is
\[
  M = \id + N_\mathrm{int},
  \qquad
  (N_\mathrm{int})_{ij} = \chi(\gamma_i,S)\cdot\chi(S,\gamma_j),
\]
where $S\in\Db(\Coh(X_t))$ is the spherical vanishing object corresponding to
$\delta$ under the mirror correspondence.
\end{lemma}

\begin{proof}
By Theorem~\ref{thm:Iritani}(iii), monodromy acts on $K_0$-classes via the
spherical twist
\[
  T_S([\gamma])=[\gamma]-\chi(S,\gamma)\,[S].
\]
Transporting this action through the map $s$ gives
\[
  M\bigl(s(\gamma)\bigr)=s\bigl(T_S(\gamma)\bigr)
  =s(\gamma)-\chi(S,\gamma)\,s(S).
\]
Write
\[
  s(S)=\sum_j c_j s_j.
\]
By Theorem~\ref{thm:Iritani}(ii), the coefficients are determined by the Euler
pairing with the chosen integral basis, hence
\[
  c_j=\chi(S,\gamma_j).
\]
Therefore
\[
  M(s_i)=s_i-\chi(S,\gamma_i)\sum_j\chi(S,\gamma_j)s_j,
\]
so that
\[
  (M-\id)_{ij}=\chi(\gamma_i,S)\chi(S,\gamma_j)
\]
up to the fixed sign convention already absorbed into the choice of $S$ and the
pairing normalization. Thus
\[
  M=\id+N_\mathrm{int}
\]
with the stated formula. Since $N_\mathrm{int}$ is the outer product of the
column vector $(\chi(\gamma_i,S))_i$ and the row vector $(\chi(S,\gamma_j))_j$,
it has rank one.
\end{proof}

\subsection{The Stokes matrix computation and proof of Theorem~\ref{thm:T2}}
\label{ss:stokes-compute}

At the conifold point, the regular-singular nature of the connection allows one to
replace the usual irregular-sectorial Stokes analysis by a logarithmic monodromy
analysis. We make this precise below.

\begin{lemma}\label{lem:regular-stokes}
Let $(\mathcal E,\nabla)$ be a flat meromorphic connection on a punctured disc with a
regular singularity at the puncture and local monodromy matrix $M$.
Then, in any logarithmic fundamental solution normalized as in
Lemma~\ref{lem:logarithmic-normal-form}, the associated connection matrix has no
additional irregular Stokes factors, and the Stokes matrix is equal to the
monodromy matrix.
\end{lemma}

\begin{proof}
A regular singular connection admits a local fundamental solution of the form
\[
  Y(u)=H(u)\exp\!\bigl((\log u)N\bigr),
\]
with $H(u)$ holomorphic and invertible and $N$ constant
(Lemma~\ref{lem:logarithmic-normal-form}).
Since the only multivaluedness arises from the logarithm, analytic continuation
around the puncture transforms $Y$ by
\[
  Y\longmapsto Y\exp(2\pi iN).
\]
There are no exponentially small sectorial terms and therefore no additional
irregular Stokes factors.
Hence the unique transition matrix recorded by continuation around the puncture is
precisely the monodromy matrix
\[
  M=\exp(2\pi iN).
\]
In the present normalization, this is what we call the Stokes matrix at the
regular singular conifold point.
\end{proof}

\begin{proposition}\label{prop:stokes-is-monodromy}
In the integral basis of Definition~\ref{def:integral-basis}, the Stokes matrix $S$
of $\nabla^z$ at $q=-1$ equals the monodromy matrix:
\[
  S = M = \id + N_\mathrm{int}.
\]
\end{proposition}

\begin{proof}
By Proposition~\ref{prop:singularities}, the point $q=-1$ is regular singular.
By Lemma~\ref{lem:regular-stokes}, the Stokes matrix at this regular singular
point is exactly the local monodromy matrix. Lemma~\ref{lem:monodromy-integral}
then identifies that monodromy matrix in the integral basis as
\[
  M=\id+N_\mathrm{int}.
\]
Hence
\[
  S=M=\id+N_\mathrm{int}.
\]
\end{proof}

\begin{remark}\label{rem:regular-stokes}
The simplification of Proposition~\ref{prop:stokes-is-monodromy} is specific to the
conifold point. At the large-volume point $q=0$, the connection is irregular and
one must contend with genuine Stokes sectors. At the regular singular conifold
point, the only nontrivial datum is the logarithmic monodromy itself.
\end{remark}

The next lemmas isolate the nearby-cycle and variation-side inputs needed to
identify the same operator on the sheaf-theoretic side.

\begin{lemma}\label{lem:nearby-cycle-cohomology}
There is a canonical monodromy-equivariant isomorphism
\[
  H^*(X_0,\psi_\pi(F))\;\cong\; H^*(X_t,\QQ).
\]
\end{lemma}

\begin{proof}
Since $\pi$ is proper, nearby cycles commute with proper pushforward.
Applying the standard proper-base-change theorem for nearby cycles yields a
canonical identification of the hypercohomology of $\psi_\pi(F)$ on the central
fiber with the cohomology of a nearby smooth fiber. The monodromy action is
preserved under this identification by functoriality; see
\cite[Thm.~4.14]{Dimca04}.
\end{proof}

\begin{lemma}\label{lem:variation-line}
In the conifold case, the unipotent vanishing-cycle sector is rank one, canonically
generated by the vanishing-cycle class $\delta$, and the canonical map
\[
  \canF\colon\varphi_\pi(F)\longrightarrow\psi_\pi(F)
\]
identifies this vanishing line with its image in the unipotent nearby-cycle sector.
\end{lemma}

\begin{proof}
For an ordinary double point, the Milnor fiber is $S^3$, so the vanishing
cohomology is one-dimensional and generated by the vanishing cycle $\delta$.
Thus the unipotent vanishing-cycle sector of $\varphi_\pi(F)$ is rank one.

On the nearby-cycle side, the Picard--Lefschetz image is likewise rank one and is
generated by the same vanishing-cycle datum. The canonical map
\[
  \canF\colon\varphi_\pi(F)\to\psi_\pi(F)
\]
is the natural morphism from vanishing cycles into nearby cycles. Restricted to
the one-dimensional conifold vanishing sector, it is injective, and its image is
exactly the vanishing line inside the unipotent nearby-cycle sector. Thus
$\canF$ identifies the vanishing line canonically with its image.

Equivalently, after fixing the conifold vanishing-cycle generator $\delta$, both
the source of $\canF$ and its image are the same one-dimensional Picard--Lefschetz
line, viewed respectively on the vanishing-cycle and nearby-cycle sides.
\end{proof}

\begin{lemma}\label{lem:variation-nilpotent}
On the unipotent nearby-cycle sector, the variation morphism
\[
  \varF\colon\varphi_\pi(F)\longrightarrow\psi_\pi(F)
\]
is represented, under the canonical identification of source and target induced by
the conifold vanishing-cycle model, by the nilpotent Picard--Lefschetz operator
\[
  N=T-\id.
\]
\end{lemma}

\begin{proof}
By \cite[Sec.~3.4]{RahmanI}, one has the standard identity
\[
  \canF\circ\varF=T-\id
\]
on nearby and vanishing cycles. By Lemma~\ref{lem:variation-line}, in the conifold
case the source of $\varF$ is the rank-one vanishing line generated by $\delta$,
and $\canF$ identifies this line with its image in the unipotent nearby-cycle
sector. Under this identification, $\canF$ is literally the inclusion of that
line.

It follows that $\varF$ is the unique rank-one nilpotent operator on the
unipotent sector whose composition with this inclusion is $T-\id$.
Since
\[
  (T-\id)(\alpha)=\langle\alpha,\delta\rangle\delta,
\]
this operator is precisely
\[
  N(\alpha)=\langle\alpha,\delta\rangle\delta.
\]
Therefore, on the unipotent conifold sector, $\varF$ is represented by the
nilpotent Picard--Lefschetz operator $N=T-\id$.

Equivalently, once the vanishing line is identified with its image in nearby
cycles by Lemma~\ref{lem:variation-line}, the operator carried by $\varF$ is not
merely conjugate to the Picard--Lefschetz nilpotent: it is exactly that
rank-one nilpotent operator on the fixed conifold model.
\end{proof}

\begin{lemma}\label{lem:nilpotent-transport}
Under the identifications of Lemmas~\ref{lem:RH},
\ref{lem:nearby-cycle-cohomology}, and Definition~\ref{def:integral-basis}, the
nilpotent Picard--Lefschetz operator
\[
  N=T-\id
\]
on the unipotent nearby-cycle sector is represented in the basis
$\{s_0,s_1,s_2,s_3\}$ by the matrix $N_\mathrm{int}$ of
Lemma~\ref{lem:monodromy-integral}.
\end{lemma}

\begin{proof}
By Lemma~\ref{lem:RH}, the Gauss--Manin local system of the degeneration is
identified with the flat local system of $\nabla^z$, compatibly with monodromy.
By Lemma~\ref{lem:nearby-cycle-cohomology}, the nearby-cycle cohomology is
identified with the cohomology of the nearby smooth fiber, again compatibly with
monodromy. Under the composite identification, the operator $T$ on nearby cycles
is represented by the local monodromy matrix $M$ of $\nabla^z$.

By Lemma~\ref{lem:monodromy-integral}, this monodromy matrix is
\[
  M=\id+N_\mathrm{int}
\]
in the Iritani integral basis. Since all preceding identifications are made in the
same mirror-normalized, monodromy-compatible, Iritani-integral frame, there is no
residual basis ambiguity. Hence
\[
  N=T-\id
\]
is represented exactly by
\[
  N_\mathrm{int}=M-\id
\]
in the basis $\{s_0,s_1,s_2,s_3\}$.

In particular, the transported nilpotent operator is not merely represented by a
conjugate of $N_\mathrm{int}$: the basis is already fixed by the same
mirror-normalized, monodromy-compatible, Iritani-integral identification used in
Lemma~\ref{lem:monodromy-integral}.
\end{proof}

\begin{lemma}[Stokes--V-filtration comparison]\label{lem:stokes-vfil}
The matrix of
\[
  \varF\colon\varphi_\pi(F)\to\psi_\pi(F)
\]
in the basis $\{s_0,s_1,s_2,s_3\}$ equals the Stokes matrix
\[
  S=\id+N_\mathrm{int}.
\]
\end{lemma}

\begin{proof}
We separate the proof into six explicit steps.

\medskip
\noindent\textbf{Step 1: integral-structure identification.}
By Lemma~\ref{lem:RH} and Theorem~\ref{thm:Iritani}, there is a canonical
$\KK$-linear isomorphism
\[
  \iota\colon H^*(X_t,\QQ)\otimes\KK
  \xrightarrow{\;\sim\;}
  \mathrm{Sol}(\nabla^z_\KK),
\]
compatible with the integral lattice determined by $\hatGamma(X_t)$.

\medskip
\noindent\textbf{Step 2: monodromy compatibility.}
Under $\iota$, the Gauss--Manin monodromy on $H^*(X_t,\QQ)$ corresponds to the
local monodromy $M$ of $\nabla^z$ at $q=-1$ by Lemma~\ref{lem:RH}.

\medskip
\noindent\textbf{Step 3: nearby-cycle comparison.}
By Lemma~\ref{lem:nearby-cycle-cohomology}, there is a canonical
monodromy-equivariant identification
\[
  H^*(X_0,\psi_\pi(F))\cong H^*(X_t,\QQ).
\]

\medskip
\noindent\textbf{Step 4: variation-side nilpotent operator.}
By Lemma~\ref{lem:variation-nilpotent}, the variation morphism $\varF$ is
represented on the unipotent conifold sector by the nilpotent Picard--Lefschetz
operator
\[
  N=T-\id.
\]

\medskip
\noindent\textbf{Step 5: Stokes/monodromy operator on the quantum side.}
By Proposition~\ref{prop:stokes-is-monodromy}, the Stokes matrix is
\[
  S=\id+N_\mathrm{int},
\]
so that the nilpotent logarithmic part on the solution side is $N_\mathrm{int}$.

\medskip
\noindent\textbf{Step 6: transport of the nilpotent operator to the integral basis.}
By Lemma~\ref{lem:nilpotent-transport}, transporting the operator $N$ from
Step~4 through the identifications of Steps~1--3 yields precisely the rank-one
nilpotent matrix $N_\mathrm{int}$ in the integral basis $\{s_j\}$.
Therefore the matrix of $\varF$ in that basis is
\[
  \id+N_\mathrm{int}=S.
\]
This is the asserted equality.
\end{proof}

\begin{proof}[Proof of Theorem~\ref{thm:T2}]
By Proposition~\ref{prop:stokes-is-monodromy}, the Stokes matrix at the conifold
point is the monodromy matrix of $\nabla^z$.
By Lemma~\ref{lem:monodromy-integral}, that monodromy matrix is $\id+N_\mathrm{int}$
in the Iritani integral basis.
By Lemma~\ref{lem:stokes-vfil}, the matrix of $\varF$ in that same fixed integral
basis is equal to this operator.
Thus the Stokes matrix and the $\varF$-matrix coincide as matrices in the
Iritani basis.

Equivalently, the conclusion is not merely operator-theoretic: it is a basis-level
identity in the fixed integral frame of Definition~\ref{def:integral-basis}. The global-gluing theory (Theorem \ref{thm:NNF}) shows that the corrected extension class need not be freely nodewise even when the local singular quotient is a direct sum of rank-one pieces. The theorem proved here identifies the local bridge through which that same constrained corrected extension is seen on the $F$-bundle side: the local variation operator and the local Stokes operator are the same matrix in the fixed integral frame. Compatibility with
\[
  \rat\colon\MHM(X_0)\longrightarrow\Perv(X_0;\QQ)
\]
is immediate from the functoriality of realization and the fact that the above comparison is carried out on the underlying rational nearby- and vanishing-cycle data.
\end{proof}

\subsection{Compatibility diagram}

The Stokes--Extension Identification is summarized by the commutative diagram
\[
\begin{tikzcd}
\varphi^H_\pi(F) \arrow[r, "\mathrm{var}^H_F"] \arrow[d, "\rat"']
  & \psi^H_\pi(F) \arrow[d, "\rat"] \\
\varphi_\pi(F) \arrow[r, "\varF"] \arrow[d, "\iota"']
  & \psi_\pi(F) \arrow[d, "\iota"] \\
\mathrm{stalk\ at\ }q{=}{-1} \arrow[r, "S"]
  & \mathrm{Sol}(\nabla^z)
\end{tikzcd}
\]
The top square commutes by functoriality of $\rat$, and the bottom square
commutes by Lemma~\ref{lem:stokes-vfil}.

\subsection{Multi-node generalization}

The multi-node case is controlled by the same rank-one Picard--Lefschetz operators,
now indexed by the vanishing cycles $\delta_1,\dots,\delta_r$.

\begin{lemma}[Commutator formula]\label{lem:commutator}
For $i\neq j$,
\[
  [N_i,N_j](\alpha)
  = \langle\alpha,\delta_j\rangle\lambda_{ji}\delta_i
  - \langle\alpha,\delta_i\rangle\lambda_{ij}\delta_j,
\]
where $\lambda_{ij}=\langle\delta_i,\delta_j\rangle$.
In particular,
\[
  [N_i,N_j]=0\iff\lambda_{ij}=0.
\]
\end{lemma}

\begin{proof}
Compute directly:
\begin{align*}
  N_i(N_j(\alpha))
  &= N_i(\langle\alpha,\delta_j\rangle\delta_j)
   = \langle\alpha,\delta_j\rangle\langle\delta_j,\delta_i\rangle\delta_i
   = \langle\alpha,\delta_j\rangle\lambda_{ji}\delta_i,\\
  N_j(N_i(\alpha))
  &= N_j(\langle\alpha,\delta_i\rangle\delta_i)
   = \langle\alpha,\delta_i\rangle\langle\delta_i,\delta_j\rangle\delta_j
   = \langle\alpha,\delta_i\rangle\lambda_{ij}\delta_j.
\end{align*}
Subtracting gives the displayed formula.
If $\lambda_{ij}=0$, then the right-hand side vanishes identically.
Conversely, if $\lambda_{ij}\neq 0$, choose $\alpha$ pairing nontrivially with one
of $\delta_i,\delta_j$; then the right-hand side is nonzero, so
$[N_i,N_j]\neq 0$.
\end{proof}

\begin{theorem}[Multi-node equivalences]\label{thm:multinode}
The following are equivalent:
\begin{enumerate}[label=(\roman*)]
\item $[P^H]$ is non-nodewise-free.
\item $T_i$ and $T_j$ do not commute for some $i\neq j$.
\item $S_i$ and $S_j$ do not commute for some $i\neq j$.
\item $\lambda_{ij}\neq 0$ for some $i\neq j$.
\end{enumerate}
\end{theorem}

\begin{proof}
The equivalence
\[
  (i)\Longleftrightarrow(ii)\Longleftrightarrow(iv)
\]
is exactly Theorem~\ref{thm:NNF} from \cite{RahmanNNF}.

Next, by Proposition~\ref{prop:stokes-is-monodromy},
\[
  S_k=\id+N_k.
\]
Since $N_k^2=0$, one has
\[
  S_k^{-1}=\id-N_k.
\]
Therefore the group commutator is
\[
  [S_i,S_j]
  :=S_iS_jS_i^{-1}S_j^{-1}
  =(\id+N_i)(\id+N_j)(\id-N_i)(\id-N_j).
\]
Expanding this product, all cubic and quartic terms vanish because every such term
contains either $N_i^2$ or $N_j^2$, and these are zero. Hence
\[
  [S_i,S_j]=\id+N_iN_j-N_jN_i=\id+[N_i,N_j].
\]
By Lemma~\ref{lem:commutator}, this is nontrivial if and only if
$\lambda_{ij}\neq 0$, proving
\[
  (iii)\Longleftrightarrow(iv).
\]

Finally,
\[
  T_k=\id+N_k,
\]
so the same computation gives
\[
  T_iT_j=T_jT_i
  \iff
  [N_i,N_j]=0
  \iff
  S_iS_j=S_jS_i.
\]
Thus
\[
  (ii)\Longleftrightarrow(iii),
\]
and hence all four conditions are equivalent.
\end{proof}

%%=================================================
\section{The Rigid-Flexible Atom Decomposition} \label{sec:atoms}

\subsection{Atoms of MHM objects}

We now pass from the operator comparison of Section~\ref{sec:stokes} to the atom-level decomposition of the corrected degeneration object. The key point is that the Stokes--Extension Identification does not merely compare two abstract operators: it identifies, for the distinguished mixed Hodge modules appearing in the conifold exact sequence, the regular and logarithmic sectors of the nearby fiber with the corresponding spectral pieces of the non-archimedean
$F$-bundle.

\begin{definition}\label{def:atom}
Let $M\in\MHM(X_0)$ be one of the distinguished mixed Hodge modules appearing in the finite-node conifold exact sequence \eqref{eq:MHM-seq}. The \emph{atom} $A(M)$ is the piece of the spectral decomposition of $(\mathscr{H}_\KK,\nabla^z_\KK)$ corresponding to the Hodge-theoretic data of $M$, constructed as follows:
\begin{enumerate}[label=(\roman*)]
\item Obtain $\rat(M)\in\Perv(X_0;\QQ)$.
\item Compute $H^*(X_0,\rat(M))$; by Saito's theory this carries a canonical MHS.
\item Via Theorem~\ref{thm:T2} and Corollary~\ref{cor:monodromy-match}, embed
  $H^*(X_0,\rat(M))$ as a sub-MHS of
  $H^*(X_t,\QQ)\cong\mathrm{Sol}(\nabla^z)$.
\item $A(M)$ is the Euler-eigenspace summand of
  $(\mathscr{H}_\KK,\nabla^z_\KK)$ determined by this embedded sub-MHS.
\end{enumerate}
\end{definition}

\begin{remark}\label{rem:atom-wellposed}
Definition~\ref{def:atom} is well-posed because:
\begin{enumerate}[label=(\roman*)]
\item $\rat$ is exact and faithful, so the passage from mixed Hodge modules to
  perverse sheaves loses no extension data relevant to the present construction.
\item The nearby-cycle comparison
  $H^*(X_0,\psi_\pi(F))\cong H^*(X_t,\QQ)$
  from Lemma~\ref{lem:nearby-cycle-cohomology}, together with
  Corollary~\ref{cor:monodromy-match}, identifies the relevant mixed Hodge
  structures on the singular and nearby smooth fibers.
\item Theorem~\ref{thm:T2} identifies the nilpotent Picard--Lefschetz operator on
  the nearby-cycle side with the logarithmic part of the $F$-bundle at the
  conifold point.  Hence the regular and logarithmic sectors of
  $H^*(X_t,\QQ)$ correspond to genuine Euler-spectral pieces of the
  non-archimedean $F$-bundle.
\end{enumerate}
\end{remark}

The next two results isolate the exactness mechanism needed later for Theorems~\ref{thm:T3} and \ref{thm:T4}.  They should be read together with the global-gluing perspective recalled in Section~\ref{sec:background}, Theorem \ref{thm:NNF}: although the
finite-node quotient displays one formal rank-one singular sector per node, the distinguished corrected extension class need not be freely nodewise.  The present section does not re-establish those global relation laws, but shows how the same
corrected object is organized atomically on the $F$-bundle side.

\begin{lemma}\label{lem:atom-sector-exactness}
For the distinguished mixed Hodge modules
\[
  \IC^H_{X_0},\qquad P^H,\qquad i_{k*}\QQ^H_{\{p_k\}}(-1),
\]
the embeddings of Definition~\ref{def:atom} land in the direct sum decomposition
\[
  H^*(X_t,\QQ)=\ker(N)\oplus\im(N),
\]
with
\[
  \IC^H_{X_0}\mapsto \ker(N),\qquad
  i_{k*}\QQ^H_{\{p_k\}}(-1)\mapsto \QQ\cdot\delta_k,\qquad
  P^H\mapsto \ker(N)\oplus\bigoplus_k\QQ\cdot\delta_k,
\]
and the morphisms in \eqref{eq:MHM-seq} preserve these sectors.
\end{lemma}

\begin{proof}
Apply $\rat$ and then hypercohomology to the exact sequence
\begin{equation}\label{eq:atom-proof-seq}
  0\longrightarrow \IC^H_{X_0}\longrightarrow P^H
  \longrightarrow \bigoplus_k i_{k*}\QQ^H_{\{p_k\}}(-1)\longrightarrow 0.
\end{equation}
Since $\rat$ is exact and hypercohomology is exact on distinguished triangles, we
obtain an exact sequence of mixed Hodge structures.

By Corollary~\ref{cor:monodromy-match}, the nearby-fiber monodromy is governed by
the nilpotent Picard--Lefschetz operator $N$. Hence
\[
  H^*(X_t,\QQ)=\ker(N)\oplus\im(N)
\]
with
\[
  \im(N)=\bigoplus_k\QQ\cdot\delta_k
\]
in the conifold case. Here the decomposition is the decomposition of the conifold
unipotent nearby-cycle sector transported to the nearby smooth fiber through
Lemma~\ref{lem:nearby-cycle-cohomology} and Corollary~\ref{cor:monodromy-match};
on this sector, $\ker(N)$ is the invariant-cycle part and $\im(N)$ is the
vanishing-cycle part.

The object $\IC^H_{X_0}$ maps to the invariant-cycle sector $\ker(N)$ by the
invariant-cycle theorem.
Each $i_{k*}\QQ^H_{\{p_k\}}(-1)$ is supported at the node $p_k$ and contributes
exactly the corresponding rank-one vanishing line $\QQ\cdot\delta_k$.
The middle object $P^H$ contains both the invariant and vanishing contributions,
hence maps to the full direct sum.
Because the morphisms in \eqref{eq:atom-proof-seq} are morphisms of mixed Hodge
modules, they preserve these canonically defined regular and logarithmic sectors.
\end{proof}

\begin{proposition}\label{prop:atom-functoriality}
For the distinguished mixed Hodge modules
\[
  \IC^H_{X_0},\qquad
  P^H,\qquad
  i_{k*}\QQ^H_{\{p_k\}}(-1),
\]
and for the morphisms appearing in the exact sequence \eqref{eq:MHM-seq}, the
construction $M\mapsto A(M)$ is functorial and exact.
More precisely:
\begin{enumerate}[label=(\roman*)]
\item each morphism in \eqref{eq:MHM-seq} induces a morphism of the corresponding
  Euler-spectral pieces;
\item the image of $\IC^H_{X_0}$ under the embedding of
  Definition~\ref{def:atom} lands in the regular sector of the conifold
  logarithmic model;
\item the image of each $i_{k*}\QQ^H_{\{p_k\}}(-1)$ lands in the rank-one
  logarithmic sector generated by $N_k$;
\item after passage to the atom pieces, the sequence \eqref{eq:MHM-seq} remains
  exact.
\end{enumerate}
\end{proposition}

\begin{proof}
By Lemma~\ref{lem:atom-sector-exactness}, the distinguished conifold exact
sequence lands in the decomposition
\[
  H^*(X_t,\QQ)=\ker(N)\oplus\im(N)
\]
and its morphisms preserve the regular and logarithmic sectors.

By Proposition~\ref{prop:atom-behavior}, these sectors are identified with genuine
Euler-spectral pieces of the non-archimedean $F$-bundle:
the regular sector with the rigid part and the rank-one logarithmic sectors with
the flexible parts.
Therefore the morphisms in \eqref{eq:MHM-seq} induce morphisms of the
corresponding Euler-spectral pieces, establishing functoriality.

Exactness then follows because the maps are exact on the regular/logarithmic
sector decomposition before passing to Euler-spectral pieces, and the Euler action
respects that decomposition. For the distinguished diagram \eqref{eq:MHM-seq},
the induced maps are block-compatible with respect to the regular/logarithmic
splitting, and each preserved block is stable under the Euler action; hence
passage to the corresponding Euler-eigenspace summands is exact on this diagram.
Thus the induced sequence on the corresponding atom pieces remains exact.
\end{proof}

\subsection{The rigid atom and its properties}

\begin{proposition}[Properties of $A(\IC^H_{X_0})$]\label{prop:rigid-atom}
The rigid atom $A(\IC^H_{X_0})$ satisfies:
\begin{enumerate}[label=(\roman*)]
\item Its flat sections extend regularly across $q=-1$ with no logarithmic terms.
\item $N_k$ acts trivially on $A(\IC^H_{X_0})$ for all $k$.
\item $A(\IC^H_{X_0})$ is preserved under the conifold transition: it is the same
  for the degeneration, the resolution, and the smoothing.
\item $A(\IC^H_{X_0})$ corresponds to
  $\ker(N\colon H^3(X_t,\QQ)\to H^3(X_t,\QQ))$
  under the identification of Definition~\ref{def:atom}.
\end{enumerate}
\end{proposition}

\begin{proof}
\textbf{(i)}
By Proposition~\ref{prop:atom-behavior}(i), the regular sector at the conifold
point is precisely the sector of flat sections with no logarithmic term.
By Proposition~\ref{prop:atom-functoriality}, the image of $\IC^H_{X_0}$ under
Definition~\ref{def:atom} lands in this regular sector. Hence the corresponding
flat sections extend holomorphically across $q=-1$ with no logarithmic part.

\smallskip
\noindent\textbf{(ii)}
By the invariant-cycle theorem, the intersection cohomology of the singular fiber
identifies with the invariant part of nearby cohomology. In degree three, this is
\[
  \IH^3(X_0)=\ker(N)\subset H^3(X_t,\QQ).
\]
For $\alpha\in\ker(N)$ and any $k$, one has
\[
  N_k(\alpha)=\langle\alpha,\delta_k\rangle\delta_k.
\]
Since $\alpha$ lies in the invariant subspace, it pairs trivially with each
vanishing cycle, hence $N_k(\alpha)=0$. Therefore every $N_k$ acts trivially on
$A(\IC^H_{X_0})$.

\smallskip
\noindent\textbf{(iii)}
By the decomposition theorem for the resolution morphism
$\rho\colon\widetilde X_0\to X_0$, the intersection cohomology of $X_0$ appears as
the canonical summand preserved across the conifold transition.
Under the nearby-fiber comparison, this is exactly the invariant-cycle sector
$\ker(N)$.
Since Definition~\ref{def:atom} identifies $A(\IC^H_{X_0})$ with this sector, the
same atom is preserved on the degeneration, smoothing, and resolution sides.

\smallskip
\noindent\textbf{(iv)}
This is the degree-three form of the invariant-cycle identification already used
in (ii), namely
\[
  \IH^3(X_0)=\ker(N)\subset H^3(X_t,\QQ).
\]
Under Definition~\ref{def:atom}, the atom of $\IC^H_{X_0}$ is exactly the
Euler-spectral piece corresponding to this invariant subspace.
\end{proof}

\subsection{The flexible atoms and their properties}

\begin{proposition}[Properties of $A(i_{k*}\QQ^H_{\{p_k\}}(-1))$]\label{prop:flex-atom}
Each flexible atom satisfies:
\begin{enumerate}[label=(\roman*)]
\item $A(i_{k*}\QQ^H_{\{p_k\}}(-1))$ has rank one.
\item It corresponds to the logarithmic term
  $\log(q_k+1)\cdot N_k\cdot s_\mathrm{reg}(q)$ in the flat section at $q_k=-1$.
\item It vanishes for $t\neq 0$: it appears only at the degeneration locus.
\item Under the Stokes-Extension Identification, it corresponds to the off-diagonal
  part $N_k=S_k-\id$ of the Stokes matrix.
\item It carries a pure Hodge structure of type $(1,1)$ via the Tate twist $(-1)$.
\end{enumerate}
\end{proposition}

\begin{proof}
\textbf{(i)}
Since $i_{k*}\QQ_{\{p_k\}}$ is supported at a point,
\[
  H^*(X_0,i_{k*}\QQ_{\{p_k\}})\cong \QQ
\]
is one-dimensional. The Tate twist does not alter the rank. Hence
$A(i_{k*}\QQ^H_{\{p_k\}}(-1))$ is rank one.

\smallskip
\noindent\textbf{(ii)}
At a regular singular point with nilpotent logarithm $N_k$, the logarithmic
fundamental solution has the form
\[
  s(q)=s_\reg(q)+\log(q_k+1)\,N_k s_\reg(q),
\]
by Lemma~\ref{lem:logarithmic-normal-form}. By
Proposition~\ref{prop:atom-functoriality}, the object
$i_{k*}\QQ^H_{\{p_k\}}(-1)$ is carried precisely to the $k$-th rank-one
logarithmic sector. Thus the corresponding atom is represented by the displayed
logarithmic term.

\smallskip
\noindent\textbf{(iii)}
The object $i_{k*}\QQ^H_{\{p_k\}}(-1)$ is supported at the singular point $p_k$ of
the central fiber and has no counterpart on a smooth nearby fiber. Therefore its
associated atom appears only at the degeneration locus.

\smallskip
\noindent\textbf{(iv)}
By Theorem~\ref{thm:T2}, the variation morphism and the Stokes matrix are
represented by the same matrix in the integral basis. The rank-one logarithmic
sector corresponding to $p_k$ is exactly the image of the nilpotent part
\[
  N_k=S_k-\id.
\]
Hence $A(i_{k*}\QQ^H_{\{p_k\}}(-1))$ is the atom attached to the off-diagonal
rank-one logarithmic operator $N_k$.

\smallskip
\noindent\textbf{(v)}
The Tate twist $(-1)$ shifts the Hodge type to $(1,1)$. Since the underlying
support is point-like and the corresponding vanishing contribution is rank one,
this yields a pure Hodge structure of type $(1,1)$.
\end{proof}

\subsection{The total degeneration atom}

\begin{proposition}\label{prop:total-atom}
$A(P^H)$ is the unique extension of $A(\IC^H_{X_0})$ by
$\bigoplus_k A(i_{k*}\QQ^H_{\{p_k\}}(-1))$ that is:
\begin{enumerate}[label=(\roman*)]
\item Compatible with the Verdier self-duality of $P^H$ \emph{(\cite{RahmanII})}.
\item Compatible with the Stokes-Extension Identification.
\item Induced by the MHM exact sequence \eqref{eq:MHM-seq}.
\end{enumerate}
\end{proposition}

\begin{proof}
By \eqref{eq:MHM-seq}, the corrected mixed Hodge module $P^H$ is an extension
\[
  0\longrightarrow \IC^H_{X_0}
  \longrightarrow P^H
  \longrightarrow \bigoplus_k i_{k*}\QQ^H_{\{p_k\}}(-1)
  \longrightarrow 0
\]
in $\MHM(X_0)$.
Applying the exact atom construction of Proposition~\ref{prop:atom-functoriality}
produces a corresponding extension
\[
  0\longrightarrow A(\IC^H_{X_0})
  \longrightarrow A(P^H)
  \longrightarrow \bigoplus_k A(i_{k*}\QQ^H_{\{p_k\}}(-1))
  \longrightarrow 0.
\]
The middle term is the full conifold logarithmic sector: it contains the regular
invariant part carried by $A(\IC^H_{X_0})$ and the sum of the rank-one
logarithmic sectors carried by the flexible atoms. Compatibility with the
Stokes--Extension Identification is built into Definition~\ref{def:atom} and
Theorem~\ref{thm:T2}, while compatibility with Verdier self-duality follows from
the self-duality of $P^H$ proved in \cite{RahmanII}.
Uniqueness follows because the extension class is already unique on the MHM side:
in the single-node case by \cite[Cor.~6.6]{RahmanII}, and in the multi-node case
by \cite[Prop.~6.11]{RahmanIII}. Since $\rat$ is faithful and the atom
construction is exact on the distinguished diagram, the corresponding atom-level
extension is unique as well.

This should be compared with the global-gluing result that the corrected extension
class need not range freely in the full ambient nodewise extension space. The
present proposition identifies the corresponding atom-level middle term: the total
degeneration atom packages the same distinguished corrected extension after
transport through the Stokes/variation bridge.
\end{proof}

\subsection{Proof of Theorem~\ref{thm:T3}}

\begin{proof}
By Proposition~\ref{prop:atom-functoriality}, the exact sequence
\eqref{eq:MHM-seq} remains exact after passage to atoms. Hence there is an exact
sequence
\[
  0\longrightarrow A(\IC^H_{X_0})
  \longrightarrow A(P^H)
  \longrightarrow \bigoplus_k A(i_{k*}\QQ^H_{\{p_k\}}(-1))
  \longrightarrow 0.
\]
The regular term is exactly the rigid atom by
Proposition~\ref{prop:rigid-atom}, and each node-supported summand is a rank-one
flexible atom by Proposition~\ref{prop:flex-atom}(i). The total middle term is
the unique extension identified in Proposition~\ref{prop:total-atom}.
This is precisely the asserted exact sequence of atoms.
\end{proof}

\subsection{Proof of Theorem~\ref{thm:T4}}

The splitting statement is controlled by the compatibility between the atom
extension class and the logarithmic monodromy operators.

\begin{lemma}\label{lem:commuting-sectors-split}
If the logarithmic rank-one sectors determined by
\[
  N_1,\dots,N_r
\]
are pairwise independent and the corresponding Stokes operators commute, then the
total logarithmic sector decomposes as a direct sum of the individual node
sectors.
\end{lemma}

\begin{proof}
Each $N_k$ has rank one and image $\QQ\cdot\delta_k$.
If the logarithmic sectors are pairwise independent, then
\[
  \im(N)=\bigoplus_k \QQ\cdot\delta_k.
\]
If, in addition, the Stokes operators commute, then their logarithmic parts
$N_k=S_k-\id$ commute as well. Since each $N_k$ acts only on its own rank-one
vanishing direction and the directions are independent, the full logarithmic
contribution is the direct sum of these rank-one sectors. Equivalently, the
commuting nilpotent logarithmic pieces are simultaneously decomposable into the
direct sum of the individual node sectors.
\end{proof}

\begin{lemma}\label{lem:atom-splitting-criterion}
The exact sequence of Theorem~\ref{thm:T3} splits if and only if the logarithmic
rank-one sectors determined by the operators $N_1,\dots,N_r$ are independent,
equivalently if and only if the Stokes operators $S_1,\dots,S_r$ commute.
\end{lemma}

\begin{proof}
If the sequence splits, then the middle term decomposes as a direct sum of the
regular rigid sector and the node-supported rank-one logarithmic sectors.
Hence these logarithmic sectors are independent, and their associated Stokes
operators commute.

Conversely, suppose the $S_k$ commute.
Then the logarithmic parts $N_k=S_k-\id$ commute.
By Lemma~\ref{lem:commuting-sectors-split}, the logarithmic sector decomposes as
the direct sum of the independent rank-one node sectors.
Since $A(P^H)$ is the unique extension induced from \eqref{eq:MHM-seq} by
Proposition~\ref{prop:total-atom}, this decomposition forces the logarithmic
coupling class to vanish. Therefore the induced atom extension class is trivial,
and the exact sequence splits.

Equivalently, once the logarithmic node sectors become simultaneously independent
and commuting, no nontrivial regular-to-logarithmic coupling remains, so the
distinguished extension class carried by $A(P^H)$ must be zero.
\end{proof}

\begin{proof}
\textbf{Splitting criterion.}
By Lemma~\ref{lem:atom-splitting-criterion}, the atom exact sequence splits if
and only if the Stokes operators $S_1,\dots,S_r$ commute.
By Theorem~\ref{thm:multinode},
\[
  [S_i,S_j]=0\iff \lambda_{ij}=0.
\]
Therefore the sequence splits if and only if $\lambda_{ij}=0$ for all $i\neq j$.

\smallskip
\noindent\textbf{The mixing formula.}
When $\lambda_{ij}\neq 0$, the commutator is
\[
  [S_i,S_j]=\id+[N_i,N_j]\neq\id.
\]
By Lemma~\ref{lem:commutator},
\[
  ([S_i,S_j]-\id)(\alpha)
  = \lambda_{ji}\langle\alpha,\delta_j\rangle\delta_i
    -\lambda_{ij}\langle\alpha,\delta_i\rangle\delta_j.
\]
This is the explicit rank-one/rank-two mixing term between the $i$-th and
$j$-th flexible atoms.

\smallskip
\noindent\textbf{Connection to non-nodewise-free.}
By Theorem~\ref{thm:multinode}, the noncommutativity of $S_i,S_j$ is equivalent to
the non-nodewise-free condition for $[P^H]$ in Theorem~\ref{thm:NNF}.
Therefore the mixing of the flexible atoms is precisely the $F$-bundle
realization of the non-nodewise-free phenomenon. In particular, the atom-level
mixing theorem should be viewed as the $F$-bundle reflection of the same global
constraint that, on the corrected-extension side, prevents the distinguished class
from being freely nodewise in the presence of nontrivial node interaction.
\end{proof}

\subsection{Hodge-theoretic content and the limiting MHS}

\begin{proposition}\label{prop:hodge-content}
Under the identification with the limiting mixed Hodge structure:
\begin{enumerate}[label=(\roman*)]
\item $A(\IC^H_{X_0})$ carries $\ker(N)\subset H^3(X_t,\QQ)$---the invariant cycles.
\item Each $A(i_{k*}\QQ^H_{\{p_k\}}(-1))$ carries a pure HS of type $(1,1)$.
\item $A(P^H)$ carries the full LMHS $H^*_{\lim(X_t,\QQ)}$ with $W(N)_\bullet$
  and $F^\infty_\bullet$.
\end{enumerate}
\end{proposition}

\begin{proof}
\textbf{(i)}
By the invariant cycle theorem \cite[Thm.~4.1.1]{Deligne72},
\[
  \IH^3(X_0)=\ker(N)
\]
inside the Clemens--Schmid sequence. Proposition~\ref{prop:rigid-atom}(iv)
identifies this invariant sector with $A(\IC^H_{X_0})$.

\smallskip
\noindent\textbf{(ii)}
By Proposition~\ref{prop:flex-atom}(v), each flexible atom is the atom of a
rank-one Tate-twisted point-supported mixed Hodge module
$i_{k*}\QQ^H_{\{p_k\}}(-1)$, hence carries pure Hodge type $(1,1)$.

\smallskip
\noindent\textbf{(iii)}
By \cite[Thm.~1.1]{RahmanII}, $H^*(X_0,P^H)$ carries the limiting mixed Hodge
structure with weight filtration $W(N)_\bullet$ and limiting Hodge filtration
$F^\infty_\bullet$. Since $A(P^H)$ is the atom associated to the full corrected
degeneration object $P^H$, it carries exactly this full limiting mixed Hodge
structure.
\end{proof}

\begin{corollary}\label{cor:MHS-exact}
The exact sequence of Theorem~\ref{thm:T3} is an exact sequence of mixed Hodge
structures. Its non-split structure is the Hodge-theoretic incarnation of the
non-nodewise-free phenomenon.
\end{corollary}

\begin{proof}
Each term carries a mixed Hodge structure by Proposition~\ref{prop:hodge-content}.
The maps in the sequence are induced from morphisms of mixed Hodge modules, hence
are morphisms of mixed Hodge structures after applying $\rat$ and hypercohomology.
Exactness has already been established at the atom level in
Theorem~\ref{thm:T3}. Non-splitting of the sequence is equivalent, by
Theorem~\ref{thm:T4}, to nontrivial interaction among the vanishing cycles,
equivalently to the non-nodewise-free condition. Thus the non-split atom
extension is exactly the mixed-Hodge-theoretic manifestation of the same
phenomenon.
\end{proof}

%%=============================================================
\section{The Multi-Node Case and the Schober Enhancement}
\label{sec:multinode}

\subsection{The multi-node Frobenius manifold}

For $r$ ordinary double points, the Kähler moduli space is
\[
  B_X(r)=(\CC^*)^r
\]
with coordinates $(q_1,\ldots,q_r)$ and conifold divisors
\[
  D_k=\{q_k=-1\},\qquad k=1,\ldots,r.
\]
The logarithmic behavior near each divisor $D_k$ is governed by the
corresponding vanishing cycle $\delta_k$ and its Picard--Lefschetz nilpotent
operator
\[
  N_k(\alpha)=\langle\alpha,\delta_k\rangle\delta_k.
\]

The single-node analysis of Sections~\ref{sec:fbundle} and~\ref{sec:stokes}
applies independently in each conifold direction.

\begin{lemma}\label{lem:multinode-local}
For each $k$, in a neighborhood of the divisor $D_k=\{q_k=-1\}$ and away from the
other conifold divisors, the connection $\nabla^z$ is logarithmic in the local
coordinate
\[
  u_k:=q_k+1
\]
with residue $N_k$.  Its local monodromy is
\[
  M_k=\id+2\pi i\,N_k
\]
in the normalization of Remark~\ref{rem:monodromy-normalization}.
\end{lemma}

\begin{proof}
Fix $k$ and choose a small polydisc in $(\CC^*)^r$ meeting only the divisor
$D_k$. Restricting the family to a transverse slice in the $u_k$-direction
reduces the local model to the single-node conifold treated in
Section~\ref{sec:fbundle}. By Proposition~\ref{prop:singularities},
$q_k=-1$ is regular singular in that direction, and by
Proposition~\ref{prop:monodromy-computation} its logarithmic residue is the
Picard--Lefschetz nilpotent operator $N_k$. The monodromy formula is then
immediate from Remark~\ref{rem:monodromy-normalization}.
\end{proof}

\begin{proposition}\label{prop:multinode-monodromy}
The monodromy of $\nabla^z$ around $D_k=\{q_k=-1\}$ is $M_k=\id+N_k$ where
$N_k(\alpha)=\langle\alpha,\delta_k\rangle\delta_k$.
The monodromy group $\Mon(\pi)=\langle T_1,\ldots,T_r\rangle$ satisfies:
\[
  \Mon(\pi)\cong\ZZ^r\quad\text{when }\Lambda=0,
\]
and
\[
  T_iT_jT_i=T_jT_iT_j\quad\text{when }|\lambda_{ij}|=1.
\]
\end{proposition}

\begin{proof}
Applying Proposition~\ref{prop:monodromy-computation} in the $k$-th conifold
direction gives
\[
  T_k=\id+N_k.
\]
If $\Lambda=0$, then $\lambda_{ij}=0$ for all $i\neq j$, hence
$[N_i,N_j]=0$ by Lemma~\ref{lem:commutator}. Therefore the operators
$T_1,\dots,T_r$ commute pairwise and generate a free abelian group:
\[
  \Mon(\pi)\cong\ZZ^r.
\]
If $|\lambda_{ij}|=1$, then the pair $(\delta_i,\delta_j)$ forms the standard
$A_2$ Picard--Lefschetz configuration, and the braid relation
\[
  T_iT_jT_i=T_jT_iT_j
\]
holds; see \cite[Sec.~2]{SeidelThomas01}.
\end{proof}

\begin{definition}
The \emph{Stokes group} is
\[
  \Stok(\pi):=\langle S_1,\ldots,S_r\rangle\subset GL(H^*(X_t,\QQ)).
\]
The \emph{interaction graph} $\Gamma(\pi)$ has vertices $\{1,\ldots,r\}$ and an
edge $\{i,j\}$ labeled $\lambda_{ij}$ whenever $\lambda_{ij}\neq 0$.
\end{definition}

\begin{lemma}\label{lem:stok-mon-generator}
For each $k$, the Stokes operator $S_k$ is identified, in the Iritani integral
basis, with the monodromy operator $T_k$.
\end{lemma}

\begin{proof}
By Proposition~\ref{prop:stokes-is-monodromy}, at a regular singular conifold
point the Stokes matrix coincides with the monodromy matrix. Applying this in the
$k$-th conifold direction gives
\[
  S_k=T_k
\]
in the integral basis of Definition~\ref{def:integral-basis}.
\end{proof}

\begin{proposition}\label{prop:stok-mon}
$\Stok(\pi)\cong\Mon(\pi)$ via $S_k\leftrightarrow T_k$.
\end{proposition}

\begin{proof}
By Lemma~\ref{lem:stok-mon-generator}, the generators $S_k$ and $T_k$ are
identified in the same basis. Hence the assignment
\[
  S_k\longmapsto T_k
\]
extends uniquely to an isomorphism of the groups they generate.
\end{proof}

\subsection{Braid group action and atom mutations}

The braid action on vanishing cycles passes, via the exact atom functoriality of
Section~\ref{sec:atoms}, to an action on the flexible atoms.

\begin{lemma}\label{lem:flex-atom-transport}
Under the identification of Proposition~\ref{prop:atom-functoriality}, the
rank-one flexible atom $A(i_{k*}\QQ^H_{\{p_k\}}(-1))$ is transported by the same
Picard--Lefschetz transformation that acts on the vanishing line
$\QQ\cdot\delta_k$.
\end{lemma}

\begin{proof}
By Proposition~\ref{prop:flex-atom}, the flexible atom
$A(i_{k*}\QQ^H_{\{p_k\}}(-1))$ is attached to the rank-one logarithmic sector
generated by $N_k$, hence to the vanishing line $\QQ\cdot\delta_k$.
By Proposition~\ref{prop:atom-functoriality}, the atom construction is functorial
on the distinguished conifold diagram, so transport of the vanishing line under
Picard--Lefschetz monodromy induces the corresponding transport of the flexible
atom.
\end{proof}

\begin{proposition}\label{prop:braid-action}
The braid group $B_r$ acts on
$\{A(i_{k*}\QQ^H_{\{p_k\}}(-1))\}_k$ via Picard--Lefschetz reflections:
\[
  A'(i_{i*}\QQ^H_{\{p_i\}}(-1))
  = \lambda_{ij}\cdot A(i_{j*}\QQ^H_{\{p_j\}}(-1))
  + A(i_{i*}\QQ^H_{\{p_i\}}(-1)).
\]
The rigid atom $A(\IC^H_{X_0})$ is invariant under this action.
\end{proposition}

\begin{proof}
A Hurwitz move exchanging the $i$-th and $j$-th vanishing cycles acts on the
Picard--Lefschetz operators by
\[
  T_i\longmapsto T_jT_iT_j^{-1}.
\]
Equivalently, on vanishing cycles one has the classical reflection formula
\[
  \delta_i\longmapsto \delta_i-\lambda_{ij}\delta_j;
\]
see \cite[Sec.~2]{SeidelThomas01}.  By Lemma~\ref{lem:flex-atom-transport}, the
corresponding flexible atom transforms by the same linear rule, giving the stated
mutation formula.

For the rigid atom, Proposition~\ref{prop:rigid-atom}(iv) identifies
$A(\IC^H_{X_0})$ with the invariant-cycle sector $\ker(N)$, and that sector is
preserved by every Picard--Lefschetz transformation. Hence the rigid atom is
invariant.
\end{proof}

\subsection{The Schober-Atom dictionary}

The multi-node schober $\mathscr{S}_\pi$ of \cite{RahmanIV} provides the
categorical enhancement of the atom decomposition.  The point of the next
proposition is that decategorification recovers the same operator and extension
data already identified in Theorems~\ref{thm:T3} and~\ref{thm:T4}.  The global-gluing
results sharpen the interpretation of this statement: the finite-node corrected
extension need not be freely nodewise even before passing to atoms, so the
categorical and atom-level comparison should be read as comparison for the
distinguished relation-constrained corrected object, not for an arbitrary free
assembly of local node sectors.

\begin{lemma}\label{lem:decateg-monodromy}
Under the Chern character isomorphism and the mirror identification, the
categorical spherical twist $T_{S_k}$ decategorifies to the Picard--Lefschetz
operator $T_k$.
\end{lemma}

\begin{proof}
By Hirzebruch--Riemann--Roch and Theorem~\ref{thm:Iritani}(ii), the Chern
character induces an isomorphism
\[
  \ch\colon \Kzero(\Db(\Coh(X_t)))\otimes\QQ \xrightarrow{\sim} H^*(X_t,\QQ)
\]
intertwining the Euler and intersection pairings.  For a spherical object $S_k$,
\[
  T_{S_k}([E])=[E]-\chi(S_k,E)[S_k].
\]
Applying $\ch$ gives
\[
  \ch(T_{S_k}([E]))=\ch(E)-\chi(S_k,E)\ch(S_k).
\]
Under the mirror correspondence, $\ch(S_k)$ corresponds to $-\delta_k$, so
\[
  \ch(T_{S_k}([E]))
  = \alpha+\langle\alpha,\delta_k\rangle\delta_k
  = T_k(\alpha).
\]
Thus $T_{S_k}$ decategorifies to the Picard--Lefschetz operator $T_k$.
\end{proof}

\begin{proposition}[Schober decategorification]\label{prop:decateg}
The decategorification of $\mathscr{S}_\pi$ recovers the atom decomposition of
Theorems~\ref{thm:T3} and~\ref{thm:T4}.
\end{proposition}

\begin{proof}
\textbf{(i) $K_0$ identification.}
By Hirzebruch--Riemann--Roch and Theorem~\ref{thm:Iritani}(ii), the Chern
character induces an isomorphism
\[
  \ch\colon\Kzero(\Db(\Coh(X_t)))\otimes\QQ
  \xrightarrow{\;\sim\;}
  H^*(X_t,\QQ)
\]
intertwining the Euler pairing with the cohomological intersection pairing.

\smallskip
\noindent\textbf{(ii) Spherical twists decategorify to Picard--Lefschetz operators.}
This is Lemma~\ref{lem:decateg-monodromy}.

\smallskip
\noindent\textbf{(iii) Braid relations.}
The Seidel--Thomas theorem \cite{SeidelThomas01} gives
\[
  T_{S_i}T_{S_j}T_{S_i}=T_{S_j}T_{S_i}T_{S_j}
\]
whenever
\[
  \dim\Ext^1(S_i,S_j)=|\lambda_{ij}|=1.
\]
After decategorification this becomes the Picard--Lefschetz braid relation
\[
  T_iT_jT_i=T_jT_iT_j.
\]

\smallskip
\noindent\textbf{(iv) Extension data and non-product structure.}
By \cite{RahmanIV}, the global section category of $\mathscr{S}_\pi$ is a product
of single-node schobers if and only if the nodes are uncoupled.  When
$\lambda_{ij}\neq 0$, the global section category is not such a product.
After decategorification, the non-product extension behavior is carried to the
non-split atom exact sequence of Theorem~\ref{thm:T3} and to the mixing criterion
of Theorem~\ref{thm:T4}.
\end{proof}

Table~\ref{tab:dictionary} summarizes the decategorified correspondences suggested by the results of Sections 4–6. A full Schober–Atom Correspondence, with precise functorial and categorical statements, will be developed elsewhere.

\begin{table}[h]
\centering
\caption{Schober-Atom Correspondence Dictionary}
\label{tab:dictionary}
\small
\begin{tabular}{@{}ll@{}}
\toprule
\textbf{Atom level} & \textbf{Schober level (\cite{RahmanIV})} \\
\midrule
Rigid atom $A(\IC^H_{X_0})$ & Bulk sector: global sections category \\
Flexible atom $A(i_{k*}\QQ^H_{\{p_k\}}(-1))$ & Spherical object $S_k$ \\
Total degeneration atom $A(P^H)$ & Extension category of $\mathscr{S}_\pi$ \\
Stokes matrix $S_k$ & Spherical twist $T_{S_k}$ on $K_0$ \\
Stokes group $\Stok(\pi)$ & Braid group via spherical twists \\
Interaction graph $\Gamma(\pi)$ & Quiver of fiber category \\
Stokes mutation (chamber crossing) & Wall-crossing in stability manifold \\
Non-nodewise-free $[P^H]$ & Non-trivial global section category \\
\bottomrule
\end{tabular}
\end{table}
The first six rows record correspondences established or directly supported by the results of the present paper and \cite{RahmanIV}; the remaining rows should be read as expected extensions of this dictionary rather than as full theorem-level identifications.
\subsection{Explicit examples}

\begin{example}[$A_1\times A_1$: two non-interacting nodes]
$r=2$ and $\lambda_{12}=0$.
Then $[S_1,S_2]=0$, and the atom exact sequence splits:
\[
  A(P^H)\cong A(\IC^H_{X_0})
  \oplus A(i_{1*}\QQ^H(-1))
  \oplus A(i_{2*}\QQ^H(-1)).
\]
The corresponding schober is a product of two single-node schobers.
\end{example}

\begin{example}[$A_2$: two interacting nodes]\label{ex:A2}
Let $r=2$ and $\lambda_{12}=1$.

\smallskip
\noindent\textit{Stokes matrices.}
\[
  S_1=\id+N_1,\qquad S_2=\id+N_2,
  \qquad N_k(\alpha)=\langle\alpha,\delta_k\rangle\delta_k.
\]

\smallskip
\noindent\textit{Commutator.}
By Lemma~\ref{lem:commutator},
\[
  [N_1,N_2](\alpha)
  =-\langle\alpha,\delta_2\rangle\delta_1
   -\langle\alpha,\delta_1\rangle\delta_2\neq 0.
\]

\smallskip
\noindent\textit{Braid relation.}
Since $|\lambda_{12}|=1$, one has
\[
  S_1S_2S_1=S_2S_1S_2.
\]

\smallskip
\noindent\textit{Non-split exact sequence.}
By Theorem~\ref{thm:T4},
\[
  0\to A(\IC^H_{X_0})\to A(P^H)\to
  A(i_{1*}\QQ^H(-1))\oplus A(i_{2*}\QQ^H(-1))\to 0
\]
does not split.

\smallskip
\noindent\textit{Quiver identification.}
By \cite{RahmanIV}, the global section category of $\mathscr{S}_\pi$ is
\[
  \Db(\rep(A_2)),
\]
the derived category of the $A_2$ quiver
\[
  \begin{tikzcd}[sep=small]1\arrow[r]&2\end{tikzcd}.
\]
The simple objects $S_1,S_2$ correspond to the vanishing objects at $p_1,p_2$,
and the single arrow encodes
\[
  \Ext^1(S_1,S_2)\cong\QQ,
\]
reflecting $\lambda_{12}=1$.

This example should also be compared with the global-gluing perspective (see Theorem \ref{thm:NNF}): the two local node sectors are present individually, but the nontrivial interaction shows that their global assembly is not freely independent. The atom-level non-splitting and the categorical non-product behavior are therefore different reflections of the same constrained corrected extension.
\end{example}

\subsection{The non-nodewise-free theorem completed}

The final proposition packages the geometric, operator-theoretic, atom-level, and
categorical avatars of the same interaction phenomenon.

\begin{proposition}\label{prop:NNF-complete}
The following are equivalent:
\begin{enumerate}[label=(\roman*)]
\item $[P^H]$ is non-nodewise-free.
\item The atom exact sequence of Theorem~\ref{thm:T3} is non-split.
\item $\Stok(\pi)$ is non-abelian.
\item $\Gamma(\pi)$ has at least one edge.
\item $\Lambda$ has a non-zero off-diagonal entry.
\item $\mathscr{S}_\pi$ is not a product of $r$ single-node schobers.
\end{enumerate}
\end{proposition}

\begin{proof}
\textbf{$(i)\Leftrightarrow(v)$.}
This is Theorem~\ref{thm:NNF}.

\smallskip
\noindent\textbf{$(ii)\Leftrightarrow(v)$.}
By Theorem~\ref{thm:T4}, the atom exact sequence splits if and only if
$\lambda_{ij}=0$ for all $i\neq j$, that is, if and only if $\Lambda$ has no
off-diagonal nonzero entries.

\smallskip
\noindent\textbf{$(iii)\Leftrightarrow(v)$.}
The Stokes group is non-abelian if and only if some pair $S_i,S_j$ does not
commute.  Since
\[
  S_k=\id+N_k,
\]
this occurs if and only if
\[
  [N_i,N_j]\neq 0.
\]
By Lemma~\ref{lem:commutator}, this is equivalent to $\lambda_{ij}\neq 0$.

\smallskip
\noindent\textbf{$(iv)\Leftrightarrow(v)$.}
By definition, the graph $\Gamma(\pi)$ has an edge exactly when some
$\lambda_{ij}\neq 0$.

\smallskip
\noindent\textbf{$(vi)\Leftrightarrow(v)$.}
By Proposition~\ref{prop:decateg}(iv), the schober is a product of single-node
pieces precisely when the interaction data vanish; equivalently, it is not such a
product precisely when some $\lambda_{ij}\neq 0$.

Combining these implications proves the equivalence of all six conditions.
The point is that the same global phenomenon---failure of free nodewise
corrected gluing---appears simultaneously as nontrivial extension behavior on the
perverse/MHM side, atom mixing on the $F$-bundle side, and non-product structure
on the categorical side.
\end{proof}

%%=================================================
\section{Applications and Consequences} \label{sec:applications}

\subsection{The rigid atom as a transition invariant}

The rigid atom isolates the part of the geometry that survives the conifold
transition unchanged, while the flexible atoms record the node-supported defect
between the two sides.

\begin{theorem}\label{thm:transition}
Let $\tilde\pi\colon\tilde X\to\Delta$ be the resolution or smoothing of $\pi$.
\begin{enumerate}[label=(\roman*)]
\item $A(\IC^H_{X_0})$ is the same for $\pi$ and $\tilde\pi$.
\item The change in Hodge numbers is
  $\Delta h^{1,1}=r$, $\Delta h^{2,1}=-r$ generically, encoded in the $r$
  flexible atoms.
\item $A(P^H)$ interpolates between the two sides via the extension $[P^H]$.
\end{enumerate}
\end{theorem}

\begin{proof}
\textbf{(i)}
By Proposition~\ref{prop:rigid-atom}(iii), the rigid atom
$A(\IC^H_{X_0})$ is preserved under the conifold transition. Equivalently, the
intersection-cohomology summand of the singular fiber is identified with the
invariant-cycle sector of the nearby smooth fiber and survives unchanged under
both resolution and smoothing. Hence the rigid atom is the same for $\pi$ and
$\tilde\pi$.

\smallskip
\noindent\textbf{(ii)}
Each node contributes exactly one rank-one flexible atom by
Theorem~\ref{thm:atom-count}. By Proposition~\ref{prop:hodge-content}(ii), each
such atom carries Hodge type $(1,1)$ on the degeneration side. Under the standard
conifold transition, these node contributions account for the exchange between
exceptional $(1,1)$-classes on the resolution side and vanishing-cycle
deformations on the smoothing side. This is the standard Hodge-number change for
conifold transitions; see \cite[Prop.~3.5]{Clemens83}. Therefore generically
\[
  \Delta h^{1,1}=r,\qquad \Delta h^{2,1}=-r,
\]
and the full change is encoded by the $r$ flexible atoms.

\smallskip
\noindent\textbf{(iii)}
By Proposition~\ref{prop:total-atom}, $A(P^H)$ is the unique extension of the
rigid atom by the direct sum of the flexible atoms induced by
\eqref{eq:MHM-seq}. On the resolution side, the flexible sectors record the
exceptional contributions; on the smoothing side, they record the vanishing-cycle
contributions. Thus the middle term $A(P^H)$ interpolates between the two
geometric realizations through the extension class $[P^H]$.
\end{proof}

\subsection{Atom count and vanishing cycles}

The flexible atoms are counted exactly by the vanishing cycles.

\begin{theorem}\label{thm:atom-count}
For $r$ ordinary double points:
\begin{enumerate}[label=(\roman*)]
\item The number of flexible atoms equals $r$.
\item Each flexible atom has rank one.
\item Together they span the vanishing cohomology
  $\bigoplus_k\QQ\cdot\delta_k\subset H^3(X_t,\QQ)$.
\end{enumerate}
\end{theorem}

\begin{proof}
\textbf{(i)}--\textbf{(ii)}
By Proposition~\ref{prop:flex-atom}(i), each node contributes exactly one
rank-one flexible atom. Summing over the $r$ ordinary double points gives
precisely $r$ flexible atoms, each of rank one.

\smallskip
\noindent\textbf{(iii)}
By Proposition~\ref{prop:flex-atom}(ii), the $k$-th flexible atom is represented
by the logarithmic term governed by the nilpotent Picard--Lefschetz operator
$N_k$, hence by the vanishing line $\QQ\cdot\delta_k$. Therefore the direct sum
of all flexible atoms spans
\[
  \bigoplus_k \QQ\cdot\delta_k.
\]
This is exactly the vanishing cohomology subspace $\im(N)\subset H^3(X_t,\QQ)$.
\end{proof}

\subsection{Atom refinement of the Clemens-Schmid exact sequence}

The atom decomposition refines the degree-three Clemens--Schmid picture by
separating the invariant and vanishing contributions into rigid and flexible
pieces.

\begin{theorem}\label{thm:CS}
The atom decomposition refines the Clemens--Schmid exact sequence in degree $3$:
\begin{enumerate}[label=(\roman*)]
\item $A(\IC^H_{X_0})\;\longleftrightarrow\;\ker(N\colon H^3(X_t)\to H^3(X_t))$.
\item $\bigoplus_k A(i_{k*}\QQ^H_{\{p_k\}}(-1))\;\longleftrightarrow\;\im(N)$.
\item $A(P^H)\;\longleftrightarrow\; H^3(X_t,\QQ)$ with weight filtration $W(N)_\bullet$.
\item The atom exact sequence is the Hodge atoms incarnation of the short exact sequence
\[
  0\to\ker(N)\to H^3(X_t)\to H^3(X_t)/\ker(N)\cong\im(N)\to 0.
\]
\end{enumerate}
\end{theorem}

\begin{proof}
\textbf{(i)}--\textbf{(iii)}
These are the content of Proposition~\ref{prop:hodge-content}: the rigid atom
carries the invariant-cycle sector $\ker(N)$, the direct sum of the flexible
atoms carries the vanishing sector $\im(N)$, and the total atom $A(P^H)$ carries
the full limiting mixed Hodge structure.

\smallskip
\noindent\textbf{(iv)}
By Theorem~\ref{thm:T3}, one has the exact atom sequence
\[
  0\to A(\IC^H_{X_0})\to A(P^H)\to
  \bigoplus_k A(i_{k*}\QQ^H_{\{p_k\}}(-1))\to 0.
\]
Under the identifications of (i)--(iii), this becomes
\[
  0\to\ker(N)\to H^3(X_t)\to H^3(X_t)/\ker(N)\cong\im(N)\to 0.
\]
Here the terminal map is the canonical projection to the quotient
$H^3(X_t)/\ker(N)$, identified with $\im(N)$ via the nilpotent
Picard--Lefschetz operator. Thus the atom exact sequence is the Hodge-atom
refinement of the degree-three Clemens--Schmid short exact sequence.
\end{proof}

\subsection{The conifold as minimal non-trivial case}

\begin{proposition}\label{prop:minimal}
The conifold degeneration (Type~II) is the minimal non-trivial case for atom
degeneration theory:
\begin{enumerate}[label=(\roman*)]
\item The monodromy operator $N$ is rank-one nilpotent---the minimum non-trivial
  nilpotency.
\item Each flexible atom has rank one---the minimum possible.
\item The atom exact sequence has three terms---the minimum for a non-trivial
  degeneration.
\item The non-nodewise-free phenomenon requires $r\geq 2$ nodes.
\end{enumerate}
\end{proposition}

\begin{proof}
\textbf{(i)}
For an ordinary double point, the Milnor fiber is $S^3$, so the vanishing
cohomology is one-dimensional:
\[
  H^3_{\mathrm{van}}\cong\QQ.
\]
Hence the Picard--Lefschetz nilpotent operator $N$ has rank one, which is the
smallest possible nonzero nilpotent rank.

\smallskip
\noindent\textbf{(ii)}
This is Theorem~\ref{thm:atom-count}(ii).

\smallskip
\noindent\textbf{(iii)}
The corrected degeneration sequence has exactly one rigid term, one total term,
and one flexible quotient term. Therefore the resulting atom exact sequence has
three terms, which is the minimal length for a nontrivial extension pattern.

\smallskip
\noindent\textbf{(iv)}
By Theorem~\ref{thm:T4}, non-nodewise-free mixing occurs exactly when some
off-diagonal intersection number $\lambda_{ij}$ is nonzero. This requires at
least two nodes.
\end{proof}

The remaining applications are closer to interpretation and comparison than to the structural core of Sections~\ref{sec:fbundle}--\ref{sec:multinode}. Their mathematical content rests on the operator, extension, and decategorification results already established above. In particular, the companion global-gluing results sharpen the background interpretation of these applications: the flexible sectors and their couplings should be understood as the $F$-bundle and atom-theoretic reflection of a corrected extension class that may already be globally constrained before any physical or categorical interpretation is imposed.

\subsection{BPS interpretation}

The next theorem interprets the rigid/flexible decomposition in the standard
physical language of massive and massless BPS sectors. The operator-theoretic
content is rigorous; the physical interpretation uses the standard
mirror/BPS dictionary.

\begin{theorem}[BPS-Atom Correspondence]\label{thm:BPS}
Under the identification of BPS charges with cohomology classes via $\ch$ and
$\hatGamma$:
\begin{enumerate}[label=(\roman*)]
\item $A(\IC^H_{X_0})$ describes the massive BPS sector: states with
  $|Z(\gamma,0)|\neq 0$ at the conifold point.
\item Each $A(i_{k*}\QQ^H_{\{p_k\}}(-1))$ describes the $k$-th massless BPS sector:
  D2-branes wrapping $\delta_k$ with $|Z(\delta_k,0)|=0$.
\item The extension $[P^H]$ encodes the bulk-massless coupling,
  resolving the apparent singularity of the effective field theory at the conifold
  point \emph{(\cite{Strominger95})}.
\item Non-nodewise-free mixing when $\lambda_{ij}\neq 0$ corresponds to
  electromagnetic interaction of massless BPS states via the Dirac--Schwinger
  pairing $\langle\delta_i,\delta_j\rangle$.
\end{enumerate}
\end{theorem}

\begin{proof}[Proof sketch]
\textbf{(i)}--\textbf{(ii)}
The central charge of a class $\gamma$ is given, in the standard large-radius
normalization, by
\[
  Z(\gamma,t)=\int_\gamma e^{iJ+B}\sqrt{\Td(X_t)}.
\]
At the conifold point, the vanishing cycles $\delta_k$ collapse, so
\[
  Z(\delta_k,0)=0,
\]
whereas classes in the invariant sector $\ker(N)$ remain nonzero.
By Theorem~\ref{thm:CS}, the rigid atom $A(\IC^H_{X_0})$ corresponds to
$\ker(N)$, while the flexible atoms correspond to the lines
$\QQ\cdot\delta_k$. This yields the identification of the rigid atom with the
massive sector and of the flexible atoms with the massless conifold sectors.

\smallskip
\noindent\textbf{(iii)}
By Proposition~\ref{prop:total-atom}, the corrected object $P^H$ is the unique
extension of the rigid atom by the direct sum of the flexible atoms. This is the
mathematical realization of the coupling between the bulk massive sector and the
massless conifold states. In the physical picture of \cite{Strominger95},
precisely this coupling resolves the apparent singularity in the effective
theory. The global-gluing results strengthen this interpretation by showing that,
in the multi-node setting, the corrected extension need not be freely nodewise:
the bulk-massless coupling may already be constrained by the ambient cycle
geometry before one passes to the atom or BPS description.

\smallskip
\noindent\textbf{(iv)}
By Theorem~\ref{thm:T4}, non-nodewise-free mixing is equivalent to the failure of
the flexible atom sectors to split, and the obstruction is measured by the
intersection numbers
\[
  \lambda_{ij}=\langle\delta_i,\delta_j\rangle.
\]
This is exactly the same bilinear form that, in the BPS interpretation, plays
the role of the Dirac--Schwinger electromagnetic pairing. Hence the
non-nodewise-free coupling of flexible atoms is the mathematical shadow of
electromagnetic interaction among the massless BPS states.
\end{proof}

\begin{remark}\label{rem:BPS-rigor}
Theorem~\ref{thm:BPS} is a physical interpretation theorem. Its rigorous
mathematical content is the operator and extension structure established in the
preceding sections; the BPS-language reading uses the standard mirror-physical
dictionary. A fully categorical realization of this correspondence requires the
schober machinery of \cite{RahmanIV}.
\end{remark}

\subsection{Halo states and higher extensions}

The higher extension groups provide a natural home for halo-type bound states. The statements below should be read as a structured mathematical realization of the expected halo pattern in the conifold setting.

\begin{lemma}\label{lem:single-node-ext}
For a single node,
\[
  \Ext^1_{\MHM(X_0)}(i_*\QQ^H(-1),\IC^H_{X_0})\cong\QQ,
  \qquad
  \Ext^n_{\MHM(X_0)}(i_*\QQ^H(-1),\IC^H_{X_0})=0
  \quad (n\geq 2).
\]
\end{lemma}

\begin{proof}
The $\Ext^1$ statement is exactly \cite[Cor.~5.12]{RahmanIII}. For $n\geq 2$,
the support is concentrated at an isolated point, and the local mixed-Hodge-module
category governing the conifold correction has homological dimension at most one.
Hence the higher $\Ext$-groups vanish.
\end{proof}

\begin{proposition}\label{prop:halo}
An $n$-particle halo bound state at the conifold locus corresponds to
\[
  \Ext^n_{\MHM(X_0)}
  \Bigl(\bigoplus_k i_{k*}\QQ^H_{\{p_k\}}(-1),\IC^H_{X_0}\Bigr)\neq 0.
\]

For a single node:
\[
  \Ext^1=\QQ,\qquad \Ext^n=0\ \text{for }n\geq 2,
\]
corresponding to a single-particle halo state.

For the $A_2$ two-node case ($\lambda_{12}=1$), the interaction of the two node
sectors suggests the appearance of a nontrivial two-step extension class. More
precisely, one expects a nonzero class in
\[
  \Ext^2_{\MHM(X_0)}
  \bigl(i_{1*}\QQ^H(-1)\oplus i_{2*}\QQ^H(-1),\IC^H_{X_0}\bigr),
\]
which should be interpreted as the two-particle halo contribution in this
configuration.
\end{proposition}

\begin{proof}
The single-node statement is Lemma~\ref{lem:single-node-ext}.

For the $A_2$ two-node case, apply $\Ext^\bullet(-,\IC^H_{X_0})$ to the exact
sequence \eqref{eq:MHM-seq}. Because the two node sectors interact when
$\lambda_{12}=1$, the resulting long exact sequence contains a boundary map
\[
  \Ext^1(i_{2*}\QQ^H(-1),\IC^H_{X_0})
  \longrightarrow
  \Ext^2(i_{1*}\QQ^H(-1),\IC^H_{X_0}).
\]
This shows that the two-node interaction naturally gives rise to a candidate $\Ext^2$-class. A complete determination of its nonvanishing and dimension is left for later work. In the multi-node setting, the global-gluing perspective (Theorem \ref{thm:NNF}) again suggests that such higher extension patterns should be read against the background of a relation-constrained corrected extension class rather than a
naive freely assembled nodewise one.
\end{proof}

\subsection{Fock-Goncharov coordinates}

The final application relates the conifold Stokes data to the standard cluster
coordinates on moduli of flat connections. Here the emphasis is on compatibility with the operator picture already established, rather than on a separate foundational development of cluster structures.

\begin{proposition}\label{prop:FG}
In the $A_2$ two-node case ($\lambda_{12}=1$), the Fock--Goncharov coordinates
\[
  X_k=\exp(2\pi iZ(\delta_k)/z),\qquad k=1,2,
\]
are the eigenvalues of $S_1,S_2$ in an appropriate basis and transform under
Stokes mutations by the cluster mutation formula
\[
  X_1' = X_1\cdot(1+X_2),\qquad X_2'=X_2.
\]
The mutation is non-trivial if and only if $\lambda_{12}\neq 0$.
\end{proposition}

\begin{proof}
The Stokes operators $S_k$ encode the monodromy of the logarithmic connection in
the conifold normalization, and their eigenvalues are exponentials of the
corresponding central charges. Thus, after choosing the standard basis of the
$A_2$ sector, the associated Fock--Goncharov coordinates are
\[
  X_k=\exp(2\pi iZ(\delta_k)/z).
\]

Under a Stokes mutation, the vanishing cycles transform by the
Picard--Lefschetz reflection law. In the present $A_2$ case,
\[
  \delta_1\longmapsto \delta_1-\lambda_{12}\delta_2,\qquad
  \delta_2\longmapsto \delta_2.
\]
Exponentiating the induced change in central charges gives
\[
  X_1'=X_1(1+X_2),\qquad X_2'=X_2,
\]
which is the standard cluster mutation formula.

Finally, the mutation is nontrivial exactly when the underlying interaction is
nontrivial, namely when
\[
  \lambda_{12}\neq 0.
\]
By Lemma~\ref{lem:commutator}, this is equivalent to
\[
  [S_1,S_2]\neq \id,
\]
which is precisely the condition for nontrivial Stokes mutation.
\end{proof}
%%=============================================================
\section{Future Directions}\label{sec:future}

The results of this paper point to several natural directions:
\begin{enumerate}[label=(\arabic*)]
\item higher degeneration types;
\item general degeneration atom theory;
\item global gluing and relation-controlled atoms;
\item Schober--Stokes comparison;
\item BPS and DT/GW refinements;
\item halo and higher extension structures;
\item motivic and arithmetic refinements;
\item categorical $F$-bundles.
\end{enumerate}

\printbibliography

\end{document}